\documentclass[11pt,a4paper,reqno]{amsart}

\usepackage{a4wide}

\usepackage[T2A,T1]{fontenc}
\newcommand{\CD}{\mbox{\usefont{T2A}{\rmdefault}{m}{n}\CYRD}}

\usepackage{bbm}
\usepackage[bbgreekl]{mathbbol}

\makeatletter
\def\amsbb{\use@mathgroup \M@U \symAMSb}
\makeatother

\usepackage[english]{babel}

\usepackage{graphicx}
\usepackage{xcolor}
\usepackage{indentfirst}

\usepackage[shortlabels]{enumitem}
\usepackage{todonotes}
\usepackage{comment}
\usepackage{csquotes}

\usepackage[
    style=alphabetic,
    backend=biber,
    maxbibnames=20,
    url=false
]{biblatex}

\DeclareFieldFormat[article]{title}{\mkbibemph{#1}}

\DeclareFieldFormat[book]{title}{\mkbibemph{#1}}

\DeclareFieldFormat[article]{journaltitle}{#1}

\DeclareFieldFormat[article]{volume}{\mkbibbold{#1}}

\usepackage{braket}
\usepackage{mathtools}
\usepackage{amsmath,amssymb,amsfonts,amsthm}
\usepackage{mathrsfs}

\usepackage{hyperref}

\newtheorem{theorem}{Theorem}[section]
\newtheorem{corollary}[theorem]{Corollary}
\newtheorem{lemma}[theorem]{Lemma}
\newtheorem{proposition}[theorem]{Proposition}
\newtheorem{remark}[theorem]{Remark}

\theoremstyle{definition}
\newtheorem{definition}[theorem]{Definition}

\newtheorem{example}{Example}

\DeclarePairedDelimiter\abs{\lvert}{\rvert}%
\DeclarePairedDelimiter\norm{\lVert}{\rVert}%

\makeatletter
\let\oldabs\abs
\def\abs{\@ifstar{\oldabs}{\oldabs*}}
\let\oldnorm\norm
\def\norm{\@ifstar{\oldnorm}{\oldnorm*}}
\makeatother
\usepackage{multirow}
\usepackage{import}
\usepackage{xifthen}
\usepackage{pdfpages}
\usepackage{transparent}
\usepackage{float}

\usepackage{quiver}
\usepackage{tikz}
\usepackage{tikz-cd}

\numberwithin{equation}{section}  % This is the key command

\usetikzlibrary{
  arrows.meta,
  positioning
}

\usepackage{tikz}
\usetikzlibrary{arrows.meta,positioning}

\makeatletter
\newcommand*{\rom}[1]{\expandafter\@slowromancap\romannumeral #1@}
\makeatother
 
\begin{document}

\title[Spectral Measure for Multivariate Dynamical Extremes]
{Spectral Measure and Tail Process for Multivariate Dynamical Extremes}

\author[A. Corveddu]{Alessio Corveddu}

\thanks{The author was partially supported by CMUP, member of LASI, which is financed by national funds through FCT – Fundação para a Ciência e a Tecnologia, I.P., under the projects with reference UID/00144/2025 and associated DOI given by \href{https://doi.org/10.54499/UID/00144/2025}{https://doi.org/10.54499/UID/00144/2025}.
The author also acknowledges the support provided by an FCT scholarship with Reference Number UI/BD/154808/2023.}

\date{\today}

\keywords{Multivariate extremes, point processes, extremal index, tail process, spectral measure} 
\subjclass[2020]{37A50, 37A25, 37B20, 60G70, 60G55}

\begin{abstract}
We study the asymptotic behavior of extreme events in multivariate stationary sequences, with particular emphasis on stochastic processes generated by dynamical systems. We develop a general framework for describing the dependence structure of extremes through two complementary objects: the tail process and the spectral measure.
We first extend the enriched point-process convergence of the univariate theory to the multivariate setting, proving that the decorated point process of exceedances converges to a marked Poisson cluster process whose marks retain the componentwise and temporal structure of extreme clusters. We then construct a spectral measure for stationary sequences, which assigns weight to distinct extremal profiles and characterizes the asymptotic dependence between components, and establish a de Haan-Resnick type representation. We also relate the spectral measure to the tail process, showing that its normalized profile distribution is recovered from the angular component of the tail process. Finally, we recover the classical univariate and multivariate exceedance processes as special cases. We illustrate the theory through applications to dynamical systems, using examples in which different dynamical mechanisms give rise to distinct extremal clustering patterns and cross-dependence structures. These applications demonstrate the ability of the proposed framework to capture genuinely multivariate features of extreme events and to extract information about the underlying dynamics that is not accessible through classical exceedance processes.
\end{abstract}

\maketitle 
\tableofcontents

\section{Introduction}

The study of extreme events is of critical importance in a wide range of disciplines, where their occurrence can cause significant disruption. Examples include natural hazards such as earthquakes, storms, droughts, and pandemics, as well as human-made disasters such as industrial and transport accidents, oil spills, nuclear events, and financial crises. These phenomena often correspond to rare, abnormal states of dynamical systems, namely states in which orbits enter small critical regions of the phase space associated with extreme configurations. There are two principal approaches to analyzing such behavior: the study of limiting distributions of maxima of dynamically generated stochastic processes, and the study of point processes that count the number and intensity of high observations within a given time interval (see, e.g., Freitas et al.~\cite{freitas2020complete,freitas2025enrichedfunctionallimittheorems}).

A more recent line of inquiry extends the analysis of extreme events to the multivariate setting, where one considers events in which at least one component of a multivariate random sample attains an exceptionally high value. While the behavior of individual components is well understood within univariate extreme value theory, the central challenge in the multivariate case is to characterize the dependence between components and to understand how this dependence manifests itself in the limit. In~\cite{aimino2026multivariate}, the authors laid the foundations of multivariate extreme value theory for dynamical systems by proving the existence of a multivariate limiting extreme value law.

The present work follows the second approach. We develop a general theory of multivariate rare-event point processes (REPPs) for stationary sequences, which captures simultaneously the temporal clustering of extremes and their cross-sectional dependence.

We emphasise that, although several works in the dynamical systems literature have considered higher-dimensional observables or multidimensional point processes, a genuinely multivariate perspective on extremes has not previously been adopted in this setting. In some of these works, vector-valued observables are reduced to a univariate problem~\cite{faranda2018correlation}; in others, one considers a rare event together with a (possibly $\mathbb{R}^d$-valued) mark measuring its strength, and studies the convergence of the associated spatio-temporal point processes~\cite{pene2020spatio,zweimuller2022hitting,bansard2026fractional}. Unlike the latter, where the dependence between components of the marks is not the primary object of study, our framework places this dependence at its centre. Here, by contrast, the dependence between components is the central object of study. To the best of our knowledge, no such framework for stationary sequences with temporal dependence is currently available in the literature, and dynamical systems provide an important class of examples to which the theory applies.

Our first main contribution is a convergence theorem for multivariate REPPs. We extend the decorated point-process framework of~\cite{pedersen2018tail,freitas2025enrichedfunctionallimittheorems}, in which each exceedance time is tagged with a bi-infinite sequence describing the evolution of the process within the associated cluster, to the multivariate setting. We prove that the multivariate REPP converges to a limiting Poisson cluster process for stationary sequences admitting a tail process. In contrast to~\cite{pedersen2018tail}, we impose no prescribed tail behavior on the marginal distributions (such as regular variation). The limiting process records both the temporal structure of extreme clusters and the components contributing to each exceedance, and thus retains the extremal dependence across components.

The dependence structure of a multivariate stationary sequence can be described through its copula~\cite{aimino2026multivariate}. For i.i.d.\ sequences, an alternative description is provided by the spectral measure, a measure supported on the unit sphere (or a suitable angular space) that assigns weight to the \emph{extremal profiles}, i.e., the configurations of the components that lead to joint exceedances~\cite{segers2012max,beirlant2006statistics}. The spectral measure is intimately linked to the stable tail dependence function through the classical de~Haan--Resnick representation~\cite{de1977limit}.

Our second main contribution is the construction of a spectral measure for stationary sequences with temporal dependence. Assuming that a limiting law for the maximum exists, we prove the existence and uniqueness of the spectral measure and establish an analogue of the de~Haan--Resnick representation. We further establish a precise connection between the spectral measure and the tail process, the limiting object describing the temporal evolution of extremes within clusters: the spectral measure arises as a projection of the tail process that retains only its angular information. Conversely, the spectral measure describes how the tail process decomposes according to the extremal profiles it reaches. This yields an explicit representation of the multivariate REPP in the dynamical setting, in which the limiting point process splits into contributions from the distinct extremal profiles.

The two objects thus provide complementary descriptions of extremal dependence. The tail process captures both the cross-dependence between components and the temporal dependence within an extremal cluster, whereas the spectral measure aggregates these contributions into a finite measure on the extremal profiles and, through the de~Haan--Resnick representation, specifies how these profiles combine to form the cross-dependence structure. The spectral measure therefore offers a more tractable description of the cross-dependence, while the tail process retains the temporal information needed to resolve the full cluster structure.

The paper is organized as follows. Section~2 reviews the notions from multivariate extreme value theory used throughout, including the stable tail dependence function and the multivariate extremal index, and recalls the main result of~\cite{aimino2026multivariate} for multivariate dynamical systems. Section~3 develops the general multivariate point-process framework for stationary sequences, and Section~4 establishes the corresponding spectral description and its relation to the tail process. Section~5 presents a large class of dynamical systems whose generated sequences fall within the theory. These include non-uniformly expanding systems whose induced maps satisfy the hypotheses of the main results; for these, we show how to modify the constructions of the previous sections. Finally, Section~6 derives several consequences of the point-process convergence. In particular, we recover the previously established univariate theory as a special case and obtain an analogue of the compound-Poisson convergence result of~\cite{nandagopalan1994multivariate} for multivariate processes of exceedances, viewed as random measures on $\mathbb{N}^d$.
\section{Multivariate Extreme Value Law}
Let us consider a dynamical system \((\mathcal{X}, \mathcal{B}_{\mathcal{X}}, \mu, f)\) where \(\mathcal{X} \) is a compact manifold, \(\mathcal{B}_\mathcal{X} \) is its Borel \(\sigma\)-algebra, \(f\) is a measurable map and \(\mu \) is an invariant measure. Let \(d, m \geq 1\) be integers and let \(\bbpsi\colon \mathcal{X} \to (\mathbb{R}^m)^d\) be an observable and define the associated stochastic process \(\mathbf{X} _0, \mathbf{X} _1,\dots \) as
\begin{equation}
	\label{eq:sequence-n}
	\mathbf{X} _n \coloneqq \bbpsi \circ f^n
\end{equation}
we will use the notation \(\mathbf{X} _n = (X _{n,1}, \dots, X _{n,d})\) and \(\bbpsi = (\psi_1,\dots, \psi _d)\) when we need to refer to single components. 
For an asymptotic frequency vectors \(\bbtau = (\tau_1 ,\dots, \tau_d)\) we consider a sequence of normalizing vectors \(\mathbb{u}_n(\bbtau) = (u_{n,1}(\tau_1), \dots, u_{n,d}(\tau_d))\) such that
\begin{equation}
	\label{eq:freq-univ}
	\lim_{n \to \infty} n \mu (\abs{X_{0,j}} > u_{n,j}(\tau_j)) = \lim_{n \to \infty} n \mu(\{x \in \mathcal{X} : |\psi_j(x)| > u_{n,j}(\tau_j)\}) = \tau_j
\end{equation}
With a slight abuse of notation, we write \(\abs{\cdot}\) as the Euclidean norm on \(\mathbb{R}^m\) or \(\mathbb{R} \), and for a vector \(\mathbf{X} = (X_{1}, \dots, X_{d}) \in (\mathbb{R}^m)^d\) we write \(\abs{\mathbf{X}} \coloneqq (\abs{X_{1}}, \dots, \abs{X_{d}}) \in [0,\infty)^d\).
For vectors, the notation \(\mathbf{a} \geq \mathbf{b} \) has to be read as componentwise, so \(a_i \geq b_i\).     

Define the componentwise maximum over the first \(n\) observations as
\begin{equation}
	\label{eq:evl}
	\mathbb{M}_n = \left( \bigvee_{i = 0} ^{n-1}  \abs{X_{i,1}}, \dots,  \bigvee_{i=0} ^{n-1}  \abs{X _{i,d}}  \right) = (M_{n,1}, \dots, M_{n,d}).
\end{equation}

In the multivariate setting, the limiting distribution of \(\mathbb{M}_n\)  characterizes not only the marginal tail behavior but also the extremal dependence among the components. We recall the main definitions and results that will be used in the sequel; see also~\cite{aimino2026multivariate}.

Given \(\bbtau \in (0,\infty )^d\), admitting a \textit{multivariate extreme value law} refers to finding a  multivariate extreme value distribution function \(H\)  supported on \([0,1]^d\) such that
\[
	\lim_{n \to \infty} \mu\left( \mathbb{M}_n \leq \mathbb{u}_n (\bbtau) \right) = H(\bbtau )
\]

As said before, the main interest of multivariate analysis is understanding the dependence between the components. In order to study the limit above, the most important tool will be the \textit{stable tail dependence function} associated to the \textit{copula} of the vector \((\mathbf{X} _n)_{n \geq 0}\).

\begin{definition}
	Let \(\mathbf{Y}  = (Y_1,\dots, Y_d)\)  be a random vector defined on a probability space with d.f.\ \(\mathbb{F} \). We define the copula \(C_ \mathbb{F} \) as a multivariate d.f.\ supported on \([0,1]^d\) and such that  \(C_ \mathbb{F} (\mathbb{t} ) = \mathbb{P} \left( F_1(Y_1) \leq t_1 ,\dots, F_d (Y_d) \leq t_d\right)\),  where \((F_j)_j\) denote the marginals of \(\mathbb{F}\).
\end{definition}

The copula describes the dependence structure among the marginal distributions within the joint law. Regarding the analysis of the multivariate extreme value law and related properties, we will focus on an object that allows us to recover the copula, since it describes it completely, while at the same time appearing explicitly in such a law. We therefore introduce the stable tail dependence function as follows
\begin{definition}
	The stable tail dependence function associated to an extreme copula is defined for \(\bbtau \in [0, \infty) ^d\) as
	\[
		\Gamma(\bbtau) \coloneqq  -\log C(e^{- \bbtau})
	\]
\end{definition}
Let us denote by \(\hat{\mathbf{X} }_0, \hat{\mathbf{X} }_1,\dots \) the associated i.i.d.\ sequence of random vectors with \(\hat{\mathbf{X} }_0(x) = \bbpsi(x)\) and define the analogous component-wise maxima vector and multivariate extreme value d.f.\ \(\hat{H}\). Similarly to the univariate case, we can show that existence of the limit in~\eqref{eq:evl} is equivalent to the existence of the limit 
\begin{equation}
	\label{eq:freq}
	\lim_{n \to \infty} n \mu (\abs{\mathbf{X}_0} \nleq  \mathbb{u}_n (\bbtau)) = \lim_{n \to \infty} n \mu \left(\bigcup_{j=1}^{d}\set{\abs{X _{0,j}} > u_{n,j} (\tau_j)} \right) = \hat{\Gamma} (\bbtau),
\end{equation}
and in that case \(\hat{H}(\bbtau) = e^{-\hat{\Gamma}(\bbtau)}\), where \(\hat{\Gamma }(\bbtau)\) is the average number of large observations among the first \(n\) observations. 

Moreover, the following properties hold
\begin{equation}
	\label{eq:g_prop}
	\hat{\Gamma}(c \bbtau) = c \hat{\Gamma}(\bbtau), \quad \forall c > 0;
\end{equation}
\begin{equation}
	\label{eq:g-prop2}
	\max\{\tau_1, \dots ,\tau_d\} \leq \hat{\Gamma}(\bbtau) \leq \tau _1 + \dots + \tau_d.
\end{equation}

In the stationary case, convergence in distribution of maxima always requires some conditions that allow us to establish some kind of asymptotic independence of our stationary sequence and some that avoid accumulations of clusters - \textit{i.e.} groups of high observations - in time. For \(q \in \mathbb{N} \) and \(\bbtau \in [0,\infty )^d\) we set
\[
	A_n ^{(q)} (\bbtau) \coloneqq \set{\abs{\mathbf{X}_0}\nleq\mathbb{u}_n(\bbtau)} \cap f^{-1}(\set{\mathbb{M}_q \leq \mathbb{u}_n(\bbtau)})
\]
and \(U_n(\bbtau) = A_n ^{(0)} = \set{\abs{\mathbf{X}_0}\nleq\mathbb{u}_n(\bbtau)}\).  
Moreover, for \(B \in \mathcal{B}_{\mathcal{X} }\) an event, we write
\[
	\mathscr{W}_{s,l}(B) \coloneqq \bigcap_{i=s}^{i = s +\max\{l-1,0\}}f^{-i}(B^c), \quad \mathscr{W}_{s,l} ^c (B) = (\mathscr{W}_{s,l}(B))^c
\]
Now we state the two main conditions we will use in what follows\\

\textbf{Condition} \(\mathcal{D}(\mathbb{u} _n)\). We say that \(\mathcal{D} (\mathbb{u} _n (\bbtau))\) holds for the sequence \(\mathbf{X} _0 , \mathbf{X} _1, \dots \) if for every \(l,t,n \in \mathbb{N} \) and \(q \in \mathbb{N}_0\),
\[
	\abs{\mu \left( A_n ^{(q)} (\bbtau) \cap \mathscr{W}_{t,l} \left( A_n ^{(q)} (\bbtau)\right)\right) - \mu \left( A_n ^{(q)} (\bbtau) \right) \mu \left( \mathscr{W}_{0,l} \left( A_n ^{(q)}(\bbtau)\right)\right)} \leq \gamma(q,n,t),
\]
where \(\gamma(q,n,t)\) is decreasing in \(t\) for each \(q,n\) and, for every \(q \in \mathbb{N}_0\) there exists a sequence \((t_n)_n\) such that \(t_n = o(n)\) and \(\lim_{q \to \infty}\limsup_{n \to \infty} n\gamma(q,n, t_n) = 0 \)

\vspace{1em}

\textbf{Condition} \(\mathcal{D} ^{\prime} (\mathbb{u}_n)\). We say that \(\mathcal{D} ^{\prime} (\mathbb{u}_n)\) holds for the sequence \(\mathbf{X}_0 , \mathbf{X}_1,\dots  \) if there is a sequence \((k_n)_{n \geq 0}\) such that
\begin{equation}
	\label{eq:kn}
	k_n \to \infty \quad \text{and}\quad k_n t_n = o(n),
\end{equation}
and such that, given \(r_n = \lfloor \frac{n}{k_n} \rfloor\)  
\[
	\lim_{q \to \infty} \Delta^{(q)}(\mathbb{u}_n (\bbtau)) \coloneqq \lim_{q \to \infty}\limsup_{n \to \infty}  n \mu \left(U_n(\bbtau) \cap \mathscr{W}_{[q+1, r_n)} ^c(U_n(\bbtau))\right) = 0.
\]
This last condition uses a double limit as the original from Leadbetter, as opposed to the one used in~\cite{freitas2025enrichedfunctionallimittheorems}, with a diverging \(q_n\) sequence.\\

\noindent Although the sequence \((\mathbf{X}_n)_{n \geq 0}\)  defined in~\eqref{eq:sequence-n} is stationary, it is generally not independent. The temporal dependence of extremes is captured by the multivariate extremal index function \(\vartheta (\bbtau)\), introduced in~\cite{nandagopalan1994multivariate}.
In the univariate case, the extremal index appears as the exponent of the exponential law of~\eqref{eq:evl}, i.e. given \(\tau_j \geq 0\) and \(u_{n,j} (\tau_j)\) as in~\eqref{eq:freq-univ} 
\begin{equation}
	\lim_{n \to \infty} \mu (M_{nj} \leq u_{n,j} (\tau_j)) = e^{-\theta_j \tau_j} \quad \forall j=1,\dots, d.
\end{equation}
Indeed, it was shown that the same holds in the multivariate case.

In case each component has a well-defined extremal index \(\theta_j\), we can
consider
\[
	\bbtheta \coloneqq  (\theta_1 ,\dots , \theta_d)
\]
and the \textit{multivariate extremal index function} as
\[
	\vartheta(\bbtau) \coloneqq  \frac{\Gamma(\bbtheta\bbtau)}{\hat{\Gamma}(\bbtau)}
\]
It was shown that \(\vartheta(c \bbtau) = \vartheta(\bbtau)\) for all \(c > 0\), so we also have that

\begin{gather}
	\label{eq:g1}
	\text{(\textit{margins})}\quad \Gamma(\bbtheta e_j) = \vartheta(e_j) \hat{\Gamma}(e_j) = \vartheta(e_j) = \theta_j,\quad \forall j = 1,\dots, d;\\[8pt]
	\label{eq:g2}
	\textit{(1-homogeneity)}\quad \Gamma(c \bbtheta\bbtau) = \vartheta(c \bbtau) \hat{\Gamma}(c \bbtau) = c \vartheta(\bbtau)\hat{\Gamma}(\bbtau) \quad \forall c >0.
\end{gather}

\noindent Now we are ready to state the main result in~\cite[Theorem 2.6]{aimino2026multivariate}.
\begin{theorem}
	Let \(\mathbf{X}_0, \mathbf{X} _1, \dots \) be a stationary multivariate stochastic process generated by~\eqref{eq:sequence-n}, and for \(\bbtau \in [0, \infty )^d \setminus \{\mathbf{0}\} \), let \(\mathbb{u}_n (\bbtau)\) be such that both~\eqref{eq:freq-univ} and~\eqref{eq:freq} hold, for some \(\hat{\Gamma}(\bbtau)\). Assume further that conditions \(\mathcal{D}(\mathbb{u}_n)\) and \(\mathcal{D} ^{\prime} (\mathbb{u}_n)\) hold and that \(\vartheta(\bbtau)\) given by
	\begin{equation}
	\label{eq: eimv}
		\vartheta(\bbtau ) = \lim_{q \to \infty} \lim_{n \to \infty} \frac{\mu\left( A_n ^{(q)} (\bbtau)\right)}{\mu \left( U_n(\bbtau) \right)}
	\end{equation}
	is well-defined. Then we have
	\[
		\lim_{n \to \infty} \mu (\mathbb{M} _{n} \leq \mathbb{u} _{n}(\bbtau)) = e^{-\vartheta(\bbtau) \hat{\Gamma}(\bbtau)} = e^{-\Gamma(\bbtheta\bbtau)}.
	\]
	Moreover, the marginal univariate extremal indices \(\theta_j, j = 1,\dots, d\), can be recovered from \(\vartheta(\bbtau)\) by setting all coordinates of \(\tau \)  equal to 0 except for the \(j\)-th, and \(\Gamma\) satisfies~\eqref{eq:g-prop2}.
\end{theorem}
\section{Convergence of multivariate Point Process}
In this section we consider a generic stationary sequence \(\mathbf{X}_0, \mathbf{X}_1, \dots\) taking values in \((\mathbb{R}^m)^d\), and, mimicking the approach developed in~\cite{freitas2025enrichedfunctionallimittheorems}, we extend it to study multivariate extremes. The results from this and the following section hold under conditions \(\mathcal{D}(u_n)\) and \(\mathcal{D}'(u_n)\); in Section~\ref{sec:systems} we present a class of dynamical systems to which these results can be applied.

Throughout this section we adopt the same notation as in~\cite{freitas2025enrichedfunctionallimittheorems}, so the parallelism between the two settings will be clear.

We begin by considering the two-sided shift operator on \(((\mathbb{R}^m)^d)^{\mathbb{Z}}\), defined by
\begin{align*}
\sigma \colon ((\mathbb{R}^m)^d)^{\mathbb{Z}} &\to ((\mathbb{R}^m)^d)^{\mathbb{Z}}, \\
(\sigma(\mathbf{x}))_i &= (\mathbf{x})_{i+1}.
\end{align*}

Next, we identify the different clusters. This step is standard in extreme value theory; indeed, the usual approach for stationary sequences is to partition a sample of \(n \in \mathbb{N}\) observations into blocks of size \(r_n = \lfloor n/k_n \rfloor\), separated by gaps of size \(t_n\). Accordingly, we choose \(t_n\) and \(k_n\) such that
\[
k_n, t_n \to \infty, \qquad k_n t_n = o(n),
\]
as required by condition \(\mathcal{D}'(u_n)\). The role of conditions \(\mathcal{D}\) and \(\mathcal{D}'(u_n)\) is precisely to guarantee that the blocks defined above are asymptotically independent, with each block containing at most one cluster.

With this in place, we introduce the following notation. For every \(i \in \mathbb{N}\), let
\[
    \mathbb{u}_n^{-1}(\mathbf{X}_i) \coloneqq \left(u_{n,1}^{-1}(\abs{X_{i,1}}),\dots, u_{n,d}^{-1}(\abs{X_{i,d}})\right)\in \mathbb{R}^d.
\]
For a vector \(\mathbf{X} \in (\mathbb{R}^m)^d\), let
\[
    \nu(\mathbf{X}) = \left(\frac{X_{1}}{\abs{X_{1}}}, \dots, \frac{X_{d}}{\abs{X_{d}}}\right).
\]
Moreover, for \(u \in \mathbb{R}^d\) and \(v \in (\mathbb{R}^m)^d\) with \(d,m > 0\), we denote by
\[
    (u \odot v)_{i,j} \coloneqq u_i v_{i,j} \quad \forall\, i = 1,\dots,d,\ j = 1,\dots,m
\]
the Hadamard product, which in particular for \(u,v \in \mathbb{R}^d\) reads
\[
    u \odot v = (u_1 v_1, \dots, u_d v_d).
\]
We extend it to the product of a vector \(u \in \mathbb{R}^d\) with \(v \in (({\mathbb{R}^m})^d)^{\mathbb{Z}}\) in the same way:
\[
    (u \odot v)_j \coloneqq (u \odot v_j) \quad \forall\ j \in \mathbb{Z}.
\]

Given the stationary process \((\mathbf{X} _n)_{n \geq 0}\) and assumption $\mathcal{D}^{\prime} (\mathbb{u}_n)$, the information regarding the structure of the clusters is kept in each block of size \(r_n\). We are going to store the information regarding the frequencies of our observations in the following way, by defining for all \(i < j \in \{0,\dots n\}\)

\begin{equation}
	\label{eq:blocks}
	\mathbb{X}_{n} ^{i,j} \coloneqq (\mathbb{u}_n ^{-1} (\mathbf{X}_i) \odot \nu(\mathbf{X}_i) , \dots \mathbb{u}_n ^{-1} (\mathbf{X}_{j-1}) \odot \nu(\mathbf{X}_{j-1}) ), \quad \mathbb{X}_{n,i} \coloneqq \mathbb{X} _n ^{(i-1)r_n, ir_n},
\end{equation}
so that \(\mathbb{X}_{n,i}\)  denotes the \(i\)-th normalized block.\\

\subsection{Clusters space}

We introduce a complete metric space of sequences in which every \(\mathbb{X}_n^{i,j}\) is embedded, allowing us to study point processes on it.

Fix \(d \geq 2\) the number of observables and define the spaces,
\begin{gather}
	l_0 ^{d} = \Set{\mathbf{x}  = (x_j)_{j \in \mathbb{Z} } \in ((\mathbb{R}^m)^d)^\mathbb{Z} \colon  \lim_{\abs{j}  \to \infty} \max_{i = 1,\dots d} \abs{x_{j,i}} =0 }\\
	l_\infty ^{d} = \Set{\mathbf{x}  = (x_j)_{j \in \mathbb{Z} } \in ((\overline{\mathbb{R}^m} \setminus \{\boldsymbol{0}\})^d)^\mathbb{Z}   \colon  \lim_{\abs{j}  \to \infty} \min_{i =1,\dots d}\abs{x_{j,i}} = \infty }.
\end{gather}

We define a metric structure on \(l_0^d\) and transfer it to \(l_\infty^d\) via an isometry \(\mathcal{P} \colon l_\infty^d \to l_0^d\). It is induced by the map \(p \colon \overline{\mathbb{R}^m}\setminus\{\boldsymbol{0}\} \to \mathbb{R}^m\),
\[
p(x) \coloneqq
\begin{cases}
x/\abs{x}^2 & \text{if } x \neq \infty, \\
0 & \text{otherwise,}
\end{cases}
\]
acting componentwise:
\[
\bigl(\mathcal{P}(\mathbf{x})\bigr)_j \coloneqq \bigl( p(x_{j,1}), \dots, p(x_{j,d}) \bigr).
\]
For every \(x_j \in (\mathbb{R}^m)^d\) consider the \(L^\infty\) norm  
\[
	\norm{x_j} \coloneqq \max_{k=1,\dots d}\abs{x_{j,k}}
\]
and then for \(\mathbf{x}  \in l_0 ^d\)
\[
	\norm{\mathbf{x} }_0 \coloneqq \sup_{j \in \mathbb{Z}} \norm{x_j} = \sup_{j \in \mathbb{Z}} \max_{k=1,\dots d} \abs{x_{j,k}}
\]
And we equip \(l_0 ^d \setminus \{\underline{0}\}\) with the complete separable metric given by
\[
	d(\mathbf{x}  ,\mathbf{y} ) \coloneqq (\lVert \mathbf{x}   - \mathbf{y}  \rVert_0  \land 1) \lor \left( \frac{1}{\lVert \mathbf{x}   \rVert_0} - \frac{1}{\norm{\mathbf{y} }_0}  \right)
\]
and a complete metric on \(l_\infty ^d \setminus \{\infty \}\) is then defined via \(\mathcal{P} \)  
\[
	d(\mathbf{x} , \mathbf{y} ) \coloneqq  d(\mathcal{P} (\mathbf{x}  ), \mathcal{P} (\mathbf{y} )) \quad \forall \mathbf{x} , \mathbf{y} \in l_\infty ^d.
\]
The process \(\mathbb{X}_n^{i,j}\) embeds in \(l_\infty^d\) in a simple way: we consider the sequence whose components with indices \(i, i+1, \dots, j-1\) are given by the normalized blocks of observations, while all other components are set to \(\mathbf{\infty}\). That is, for every \(i < j\),
\[
    \mathbb{X}_n^{i,j} \coloneqq \Big( \dots, \mathbf{\infty}, \mathbb{u} _n^{-1}(\mathbf{X}_i) \odot \nu(\mathbf{X}_i), \dots, \mathbb{u} _n^{-1}(\mathbf{X}_{j-1}) \odot \nu(\mathbf{X}_{j-1}), \mathbf{\infty}, \dots \Big).
\]
We define the quotient spaces \(\tilde{l}_0 ^d = l_0 ^d / \sim\) and \(\tilde{l}_\infty ^d = l_\infty ^d / \sim  \) by the relation \(\mathbf{x} \sim \mathbf{y} \) if there is \(k \in \mathbb{Z} \) such that \(\sigma ^k (\mathbf{x}  ) = \mathbf{y} \), where \(\sigma \) is the shift operator, and we denote by \(\tilde{\pi}\) the projection from \(l_\infty ^d \) to \(\tilde{l}_\infty ^d\).

The invertibility of $\mathcal{P}$ allows us to define $\tilde{\mathcal{P}} \colon 
\tilde{l}_\infty^d \to \tilde{l}_0^d$ and pull back the 
complete separable metric to the quotient.

Finally, for every \(i =1,\dots ,d\) we define
\[
	\pi_i \colon l^d _\infty \to  l_\infty
\]
the canonical projection to its \(i\)-th component.

\begin{remark}
	From the definition of the complete metric \(d(\cdot, \cdot)\) we see that a set \(A \subset \tilde{l}_\infty ^d\), \(A\) is bounded iff for every \(\mathbf{x}\in A\) there is \(\epsilon > 0\) such that \(\norm{\tilde{\mathcal{P} }(\mathbf{x} )}_0 > \epsilon  \), i.e. iff there is a \(i = 1,\dots d\) such that \(\inf_{j\in \mathbb{Z} } \abs{x_{ij}} < \frac{1}{\epsilon}\)
\end{remark}

Before introducing the tail process, we describe the objects needed to specify the dependence structure, together with the sets that will be used to verify convergence of the point processes.\\

\noindent Let \( A \in \mathcal{B}^{\mathbb{N}_0, \mathbb{Z}} \) be an event, let \( J \) be an interval contained in \([0, \infty)\). We define

\begin{equation}
	\mathscr{W}_J(A) \coloneqq  \bigcap_{i \in J \cap \mathbb{N}_0} \sigma^{-i}(A^c), \quad
	\mathscr{W}_J^c(A) \coloneqq  (\mathscr{W}_J(A))^c = \bigcup_{i \in J \cap \mathbb{N}_0} \sigma^{-i}(A).
\end{equation}

\noindent Moreover, for \( i, j \in \mathbb{N} \) we set
\( \mathscr{W}_j(A) := \mathscr{W}_{[0,j)}(A), \
\mathscr{W}_{i,j}(A) := \mathscr{W}_{[i,j)}(A) \)\\

\noindent For the event \( A \in \mathcal{B}^{\mathbb{N}_0, \mathbb{Z}} \) and \( j \in \mathbb{N} \),
\begin{align}
	& A^{(j)} := A \cap \sigma^{-1}(A^c) \cap \cdots \cap \sigma^{-j}(A^c), 
\end{align}
and for \( j = 0 \) we simply define \( A^{(0)} := A \).\\
\noindent Let
\begin{multline}
    \label{eq:F-set}
    \mathscr{F} \coloneqq \Bigl\{\left\{\mathbf{x}  = (x_j)_{j \in \mathbb{Z} } \in \left((\mathbb{R}^m)^d\right)^\mathbb{Z}  \colon x _j \in (H_{j,1} \times \dots \times H_{j,d}), j = 0,\dots N \right\} \\
	\colon H_{j,i} \in \mathscr{F}_{\mathcal{V}}\ \forall i = 1,\dots d, \ \forall j = 0,\dots N,\ N\in \mathbb{N} \Bigr\}
\end{multline}
where $\mathscr{F}_\mathcal{V} $ denotes the field generated by sets of the form $[e_{1}, f_{1}) \times \dots \times [e_{m}, f_{m})$. \\

For each $l = 1, \ldots, p \in \mathbb{N}$ suppose
$A_l \in \mathscr{F}$ and $J_l = [a_l, b_l)$,
where $0 \le a_1 < b_1 \le a_2 < b_2 \le \cdots \le a_p < b_p \le 1$.
Then for each $n \in \mathbb{N}$ define
\begin{equation}
	\tilde{J}_{k_n,l} := k_n J_l, \qquad
	J_{n,l} := \left( [(k_n a_l - 1) r_n,\ (k_n b_l + 1) r_n ) \right),
\end{equation}

\begin{equation}
	\label{eq: A-n}
	A_{n,l} \coloneqq \Set{\left(\mathbb{u}_n ^{-1}(\mathbf{X}_j) \odot \nu(\mathbf{X}_j) \right)_{j \in \mathbb{Z} } \in \left((\mathbb{R}^m) ^d\right)^\mathbb{Z} \colon \left(\mathbb{u}_n ^{-1}(\mathbf{X}_j) \odot \nu(\mathbf{X}_j) \right)_j \in A_l}
\end{equation}

\noindent For \(A \in \mathscr{F}\), we define
\begin{equation}
	\label{eq:A-set-0}
	\tilde{A} = \set{\tilde{\mathbf{x} } \in \tilde{l}_\infty \colon \tilde{\pi }^{-1} (\tilde{\mathbf{x} })\cap A \neq \emptyset}; \quad \tilde{\mathscr{J}} = \set{\tilde{A}\colon A \in \mathscr{F}}.
\end{equation}
The sets belonging to $\tilde{\mathscr{J}}$ are fundamental, since, as shown in \cite{freitas2025enrichedfunctionallimittheorems}, they are convergence determining for both weak and weak$^\#$ convergence.
Also define
\begin{gather}
	\label{eq:A-bold}
	\mathbb{A} = \set{\mathbf{x}  \in l_\infty ^d\ \colon \tilde{\pi }(\mathbf{x} )\in \tilde{A}  } = \Set{\mathbf{x} \in l_\infty ^d \ \colon \mathbf{x}  \in \bigcup_{j \in \mathbb{Z}} \sigma^{-j}(A)}\\
	\mathscr{J} = \set{\mathbb{A} \colon A \in \mathscr{F} }.
\end{gather}

\noindent Finally, we describe the sets
\begin{equation}
    U(\bbtau) = \set{\mathbf{x} \in l_\infty^d \colon \mathbf{x}_0 \in L(\bbtau)}
\end{equation}
where \(L(\bbtau) = \set{x \in (\mathbb{R}^m)^d \colon (\abs{x_1}, \dots, \abs{x_d}) \ngeq \bbtau}\) and
\[
    U_n(\bbtau) = \set{\mathbb{u}_n^{-1}(\mathbf{X}_0) \odot \nu(\mathbf{X}_0) \in L(\bbtau)} = \bigcup_{i=1}^{d} \set{\mathbb{u}_{n,i}^{-1}(\abs{\mathbf{X}_{0,i}}) < \tau_i} = \set{\mathbb{u}_n^{-1}(\mathbf{X}_0) \ngeq \bbtau}.
\]

We now slightly modify conditions \(\mathcal{D}(\mathbb{u}_n)\) and \(\mathcal{D}^{\prime}(\mathbb{u}_n)\) to be tailored to these sets, so that in what follows we work with the following conditions.\\

\textbf{Condition} \(\mathcal{D}\). We say that \(\mathcal{D}\) holds for the sequence \(\mathbf{X} _0, \mathbf{X} _1,\dots \) if for every \(m,t,n \in \mathbb{N} \), for every \(J_l\) and \(A_l\) with \(l = 0,\dots, p\) chosen as above, we have 
\[
	\abs{ \mu \left( A_{n,l} ^{(q)}  \cap \bigcap_{l=1}^p \mathscr{W}_{J_{n,l}} \left( A_{n,l} ^{(q)}  \right) - \mu \left( A_{n,l} ^{(q)}  \right) \mu \left(\mathscr{W}_{J_{n,l}} \left( A_{n,l} ^{(q)} \right)\right) \right)} \leq \gamma(q,n,t),
\]
where \(min\{J_{n,l} \cap \mathbb{N}_0\} \geq t\) and \(\gamma(q,n,t)\) is decreasing in \(t\) for each \(q,n\) and, for every \(q \in \mathbb{N}_0\) there exists a sequence \((t_n)_n\) such that \(t_n = o(n)\) and \(\lim_{q \to \infty}\limsup_{n \to \infty} n\gamma(q,n, t_n) = 0 \).\\

As explained in~\cite{freitas2025enrichedfunctionallimittheorems, aimino2026multivariate} condition \(\mathcal{D}\) is weaker than the usual mixing conditions for stationary sequences and that makes it easy to check for many dynamical systems, as it follows from a sufficiently fast decay of correlations. 

\textbf{Condition} \(\mathcal{D} ^{\prime} \). We say that \(\mathcal{D} ^{\prime}\) holds for the sequence \(\mathbb{X}_0 , \mathbb{X}_1,\dots  \) if there is a sequence \((k_n)_n\) such that
\begin{equation}
	k_n \to \infty \quad \text{and}\quad k_n t_n = o(n),
\end{equation}
and such that
\[
	\lim_{q \to \infty} \Delta^{(q)} \coloneqq \lim_{q \to \infty}\limsup_{n \to \infty} n \mu \left( U_n(\bbtau) \cap \mathscr{W}_{[q+1, r_n)} ^c (U_n (\bbtau)) \right) = 0.
\]

\subsection{Transformed anchored tail process}

In this section we define the tail process and the transformed anchored tail process, which are the main objects allowing us to describe the limiting behavior of the point processes to be introduced in the next section. The definition of such a process is similar to the one given in~\cite{pedersen2018tail}, where the authors show it to be equivalent to anisotropic regular variation in each component, a property called vector scaling regular variation (VSRV). It is worth noting that our definition has the advantage of being more general, as it does not require heavy-tailed marginals; however, its existence does not follow from any simpler, readily verifiable property, and must therefore be checked on a case-by-case basis.
Since it has nevertheless been shown to be equivalent to VSRV, our work does not lose any generality by taking the existence of the tail process as a starting point.\\

Before defining it, we fix some simple notation that will be used frequently in the sequel. We write $\bbtau = (\tau_1, \ldots, \tau_d)$ with $\tau_j \in \mathbb{R}_+$ 
for each $j = 1, \ldots, d$, and $\bbtau^* = (\tau, \ldots, \tau)$ with $\tau \in \mathbb{R}_+$,
As in~\cite{freitas2025enrichedfunctionallimittheorems}, we assume the existence of a process \((\mathbf{Y}_j)_{j\in\mathbb{Z}}\in l^d _{\infty }\) satisfying
\begin{enumerate}[(i)]
	\item \(
	      \mathcal{L}\left(\frac{1}{\tau} \mathbb{X}_n ^{r_n+s, r_n+t}\ \middle\vert \ \mathbb{u}_n ^{-1} (\mathbf{X}_{r_n}) \ngeq  \bbtau^*\right) \to \mathcal{L}((\mathbf{Y} _j)_{j = s,\dots t})\quad \text{for all} \  \tau>0, s<t, s,t \in \mathbb{Z};
	      \)
	\item the process \((\Theta_j)_{j\in \mathbb{Z}} \) defined as \(\Theta_j \coloneqq \frac{\mathbf{Y}  _j}{\min _{i = 1,\dots, d} \abs{Y _{0,i}}}\) is independent of \(\min _{i = 1,\dots, d} \abs{Y _{0,i}} \);
	\item \(\lim_{\abs{j} \to \infty} \min_{i = 1,\dots d}\abs{Y _{j,i}} = \infty \) a.s;
	\item \(\mu \left( \inf_{j\leq -1} \min_{i = 1,\dots d}\abs{Y _{j,i}} \geq 1\right) > 0\),
\end{enumerate}
where \(r_n = \lfloor \frac{n}{k_n} \rfloor\) is as in the definition of \(\mathcal{D}\) and \(\mathcal{D}^{\prime} \).

We finally define the transformed anchored tail process conditioning to the fact that \(\mathbf{Y}_0\) marks the beginning of a cluster. To do so
we choose as an anchoring map \(l_\infty ^d \to l_\infty ^d\), the first exceedance map
\[
	A(\mathbf{x}) \coloneqq \min \{j \in \mathbb{Z} : \min_{i = 1,\dots d}\abs{x_{j,i}} < 1\}
\]
Now we define the transformed anchored tail process\footnote{As noted in~\cite[Lemma 3.5.]{basrak2021compound}, if \(\mu \left( \mathbf{Y}  \in l_\infty ^d  \right) = 1\) then any anchoring gives rise to a transformed anchored processes with the same distribution.
} as
\[
	\mathbf{Z} \overset{\mathcal{L}}{=} \mathbf{Y} \mid A(\mathbf{Y} ) = 0
\]

\begin{lemma}
	\label{lemma:uniformY0}
	The random variable \(\min_{i = 1,\dots ,d} \abs{Y_{0,i}} \) is uniformly distributed on \((0,1)\). 
\end{lemma}

\begin{proposition}
	\label{prop: tetaY}
	Consider \(\vartheta(\bbtau)\) as in~\eqref{eq: eimv} and \(\mathbf{Y} \) defined above, then given \(\bbalpha \coloneqq (\alpha_1, \dots, \alpha_d)\) with 
	\[
		\alpha_i = \frac{\tau_i}{\tau_1 + \dots + \tau_d} \quad \forall i = 1, \dots, d;
	\] 
	we have 
\begin{equation}
\begin{aligned}
\vartheta(\bbtau)
&= \mu \left(
\bigcap_{i=1}^{d} \bigcap_{j=1}^{\infty}
\set{\abs{Y_{j,i}} > \alpha_i}
\ \middle|\ 
\bigcup_{i=1}^{d} \{\abs{Y_{0,i}} < \alpha_i\}
\right) \\
&= \mu \left(
\bigcap_{i=1}^{d} \bigcap_{j=-\infty}^{-1}
\set{\abs{Y_{j,i}} > \alpha_i}
\ \middle|\ 
\bigcup_{i=1}^{d} \{\abs{Y_{0,i}} < \alpha_i\}
\right).
\end{aligned}
\end{equation}
	Moreover, we also have 
	\begin{equation}
		\vartheta(\bbtau^*) = \vartheta(1,\dots, 1) = \mu  \left(\inf_{j \leq - 1}\min_{i = 1,\dots, d} \abs{Y_{j,i}}\geq 1\right) = \mu  \left(\inf_{j \geq 1}\min_{i = 1,\dots, d} \abs{Y_{j,i}}\geq 1\right). 		
	\end{equation}
\end{proposition}

\begin{proof}
	By stationarity and the definition of \(\mathbf{Y} \) we have, denoting \(\tau_1 + \dots + \tau _d = \hat{\tau } \),   
	\begin{align*}
		& \vartheta(\bbtau) = \lim_{q \to \infty} \lim_{n \to \infty} \frac{\mu (A_n ^{(q)}(\bbtau))}{\mu (A_n ^{0}(\bbtau))} = \\
		& = \lim_{q \to \infty} \lim_{n \to \infty} \frac{\mu\left(  \mathbb{u}_n ^{-1} (\mathbf{X}_{r_n}) \ngeq \bbtau, \mathbb{u}_n ^{-1} (\mathbf{X} _{r_n +1}) \geq \bbtau, \dots, \mathbb{u}_n ^{-1}  (\mathbf{X}_{r_n +q}) \geq \bbtau\right)}{\mu \left(\mathbb{u}_n ^{-1} (\mathbf{X}_{r_n}) \ngeq \bbtau\right)} = \\ 
		& = \lim_{q \to \infty} \lim_{n \to \infty} \frac{\mu\left( \frac{1}{\hat{\tau}} \mathbb{u}_n ^{-1} (\mathbf{X}_{r_n}) \ngeq \bbalpha \cap \bigcap_{j=1}^{q}  \set{\frac{1}{\hat{\tau }} \mathbb{u}_n ^{-1} (\mathbf{X} _{r_n +j}) \geq \bbalpha} \ \middle\vert \ \mathbb{u}_n ^{-1}(\mathbf{X}_{r_n}) \ngeq \hat{\bbtau}^* \right)}{\mu \left( \frac{1}{\hat{\tau }}\mathbb{u}_n ^{-1} (\mathbf{X}_{r_n}) \ngeq \bbalpha \ \middle\vert \ \mathbb{u}_n ^{-1}(\mathbf{X}_{r_n}) \ngeq \hat{\bbtau}^* \right)}; 
	\end{align*} 
	which by definition of \(\mathbf{Y} \) converges to  (as \(n \to \infty \))
	\begin{align*}
		& \lim_{q \to \infty} \frac{\mu \left(\bigcup_{i=1}^{d} \{\abs{Y_{0,i}} < \alpha_i \}, \bigcap_{i=1}^{d} \bigcap_{j=1}^{q} \set{ \abs{Y_{j,i}} > \alpha_i }\right)}{\mu \left(\bigcup_{i=1}^{d} \{\abs{Y_{0,i}} < \alpha_i \}\right)} = \\ 
		& \lim_{q \to \infty} \mu \left(\bigcap_{i=1}^{d} \bigcap_{j=1}^{q} \set{ \abs{Y_{j,i}} > \alpha_i  } \ \middle\vert \ \bigcup_{i=1}^{d} \{\abs{Y_{0,i}} < \alpha_i \}\right) = 
		\mu \left(\bigcap_{i=1}^{d} \bigcap_{j=1}^{\infty } \set{ \abs{Y_{j,i}} > \alpha_i  } \ \middle\vert \ \bigcup_{i=1}^{d} \{\abs{Y_{0,i}} < \alpha_i \}\right).
	\end{align*}
	
	To prove the other equality for \(\vartheta(1,\dots 1)\) we simply notice that 
	 \begin{align*}
		& \vartheta(\mathbb{1}) = \lim_{q \to \infty} \lim_{n \to \infty} \frac{\mu\left(  \mathbb{u}_n ^{-1} (\mathbf{X}_{r_n}) \ngeq \mathbb{1}, \mathbb{u}_n ^{-1} (\mathbf{X} _{r_n +1}) \geq \mathbb{1}, \dots, \mathbb{u}_n ^{-1}  (\mathbf{X}_{r_n +q}) \geq \mathbb{1}\right)}{\mu \left(\mathbb{u}_n ^{-1} (\mathbf{X}_{r_n}) \ngeq \mathbb{1}\right)} = \\ 
		& = \lim_{q \to \infty} \lim_{n \to \infty} \mu\left(\mathbb{u}_n ^{-1} (\mathbf{X}_{r_n}) \ngeq \mathbb{1} \cap \bigcap_{j=1}^{q}  \set{ \mathbb{u}_n ^{-1} (\mathbf{X} _{r_n +j}) \geq \mathbb{1}} \ \middle\vert \ \mathbb{u}_n ^{-1}(\mathbf{X}_{r_n}) \ngeq \mathbb{1}\right); 
	\end{align*} 
	which converges to \(\mu  \left(\inf_{j \geq 1}\min_{i = 1,\dots, d} \abs{Y_{j,i}}\geq 1\right).\) 
\end{proof}

\begin{proposition}
	\label{prop:gammaY}
	For every \(d \geq 2\), 
	\[
		\hat{\Gamma }(\mathbf{1}) = \frac{d}{1 + \sum _{\substack{m \in \{0,1\}^d \\ \abs{m}>1}} (-1)^{\abs{m}} \mu  \left(\abs{Y_{0,j}} \mathbbm{1}_{m_j = 1} < 1, j = 1,\dots, d \right)} 
	\] 
\end{proposition}

\begin{proof}
	We deal with the case \(d = 2\) for simplicity; the general case follows from an inclusion-exclusion argument. By stationarity we have
	\begin{align*}
		&\hat{\Gamma }(\mathbf{1})  = \lim_{n \to \infty} n \mu (\mathbb{u}_n ^{-1} (\mathbf{X}_{0}) \ngeq \mathbf{1} ) = \\
		&= \lim_{n \to \infty}\left( n \mu (X_{r_n,1} > u_{n,1}(1)) + n \mu (X_{r_n,2} > u_{n,2}(1)) - n \mu (X_{r_n,1} > u_{n,1}(1), X_{r_n,2} > u_{n,2}(1)) \right)
	\end{align*}
	From~\eqref{eq:freq-univ} the first two terms converges to 1, so we focus on the last term. We have 
	\begin{align*}
		& \lim_{n \to \infty}  n \mu (X_{r_n,1} > u_{n,1}(1), X_{r_n,2} > u_{n,2}(1)) = \lim_{n \to \infty}  n \mu (X_{r_n,1} > u_{n,1}(1), X_{r_n,2} > u_{n,2}(1) \cap \mathbb{u}_n ^{-1} (\mathbf{X}_{r_n}) \ngeq \mathbf{1} )= \\
		& \lim_{n \to \infty} n \mu (\mathbb{u}_n ^{-1} (\mathbf{X}_{r_n}) \ngeq \mathbf{1} ) \mu (X_{r_n,1} > u_{n,1}(1), X_{r_n,2} > u_{n,2}(1) \mid \mathbf{X}_{r_n} \nleq u_n (\mathbf{1})) ,  
	\end{align*}
	which by definition of \(\hat{\Gamma }\) and \(\mathbf{Y} \) converges to  
	\[
		\hat{\Gamma }(1,1) \cdot \mu \left(\cap_{i = 1,2} \set{\abs{Y_{0,i}} < 1}\right)
	\]  
	so putting everything together we have
	\[
		\hat{\Gamma }(1,1) = 2 - \hat{\Gamma }(1,1) \cdot \mu \left(\cap_{i = 1,2} \set{\abs{Y_{0,i}} < 1}\right)
	\]
	which gives the desired result.
\end{proof}

These proofs are useful because they allow us to express both the temporal and spatial dependence indicators, \(\vartheta(\bbtau)\) and \(\hat{\Gamma}\), solely in terms of the tail process \(\mathbf{y}\). Moreover,

$$
\Gamma(\bbtheta)=\vartheta(\mathbf{1})\hat{\Gamma}(\mathbf{1}).
$$

As we will see below, this quantity represents the intensity of the limiting point process and allows us to express the spectral measure \(H\) as a suitable projection of the tail process.\\

The following results are direct consequences of~\cite{freitas2025enrichedfunctionallimittheorems}, with only minor modifications, since the condition \(\mathcal{D}^{\prime}\) has been relaxed to allow \(q\) to be independent of \(n\). For completeness, we include the proofs, with the necessary adaptations.

\begin{proposition}
	\label{prop 3.11}
	Let \(A_1 \in \mathscr{F} \). Under \(\mathcal{D}^{\prime} \), for \(\mathbb{X}_{n,1}\) as in~\eqref{eq:blocks} we have
	\[
		\lim_{q \to \infty} \lim_{n \to \infty}\vert k_n \mu  (\mathbb{X} _{n,1} \in \mathbb{A} _1) - n \mu  (A_{n,1}^{(q)})\vert  = \lim_{q \to \infty} \lim_{n \to \infty}\vert k_n \mu  (\mathscr{W} ^c _{r_n} (A_{n,1})) - n \mu  (A_{n,1} ^{(q)}) \vert = 0,
	\]
	where \(A_{n,1}\) and \(\mathbb{A}_1\) are defined as in~\eqref{eq: A-n} and~\eqref{eq:A-bold}.
\end{proposition}

\begin{proof}
	Following the proof of~\cite[Proposition 3.11]{freitas2025enrichedfunctionallimittheorems}, we obtain this bound, which now depends on \(q\) 
	\[
		\vert k_n \mu (\mathbb{X}_{n,1} \in \mathbb{A} _1) - n \mu (A_{n,1}^{(q)}) \vert \leq 2q k_n \mu \left(A_{n,1}\right) + n \mu (A_{n,1} ^{(q)} , \mathscr{W}_{q+1, r_n} ^c (A_{n,1})).
	\]
	The second term of the right-hand side goes to zero by conditions \(\mathcal{D}^{\prime} \), while for the first one we observe that \(A_{n,1} \subseteq \set{\mathbf{X} _0 \nleq  \mathbb{u}_n (\mathbb{h})} \) where  \(\mathbb{h} = (h_1,\dots h_d)\) and  \(h_i \coloneqq \inf \{ \abs{x} \colon x \in H_{0,i}\}\) so that 
	\[
		2q k_n \mu (A_{n,1}) \leq 2q k_n \mu (\mathbf{X} _0 \nleq \mathbb{u}_n (\mathbb{h})) = \frac{2 q k_n}{n} n \mu (\mathbf{X} _0 \nleq \mathbb{u}_n (\mathbb{h})) \sim \frac{2k_n q}{n} \hat{\Gamma }(\mathbb{h} )
	\]
	and this goes to \(0\) when \(n\) goes to \(\infty \), since \(k_n = o(n)\).    
\end{proof}

\begin{corollary}
	\label{SM:cor}
	Under \(\mathcal{D}^{\prime} \),
	\[
		\lim_{n \to \infty} k_n \mu (\mathscr{W}_{r_n} ^c(U_n(\bbtau))) = \Gamma(\bbtheta \bbtau) =\vartheta(\bbtau) \hat{\Gamma}(\bbtau)
		\quad \text{and } \quad \lim_{n \to \infty} \frac{\mu (\mathscr{W}^c _{r_n}(U_n(\bbtau)))}{r_n \mu (U_n(\bbtau))} = \vartheta(\bbtau)
	\]
\end{corollary}

\begin{proof}
	By propositions~\ref{prop 3.11},~\eqref{eq: eimv} and~\eqref{eq:freq}
	\begin{align*}
		\lim_{n \to \infty} k_n \mu (\mathscr{W}^c _{r_n} (U_n(\bbtau))) = \lim_{q \to \infty} \lim_{n \to \infty} n\mu (U_n(\bbtau)^{(q)}) & = \lim_{q \to \infty} \lim_{n \to \infty} n \mu (U_n (\bbtau)) \frac{\mu (U_n ^{(q)}(\bbtau))}{\mu (U_n(\bbtau))}= \\
		& = \vartheta(\bbtau) \hat{\Gamma }(\bbtau) = \Gamma(\bbtheta \bbtau)
	\end{align*}
	and
	\begin{align*}
		\lim_{n \to \infty} \frac{\mu (\mathscr{W}^c _{r_n}(U_n(\bbtau)))}{r_n \mu (U_n(\bbtau))} = \lim_{q \to \infty} \lim_{n \to \infty} \frac{\mu(U_n ^{(q)}(\bbtau))}{\mu (U_n(\bbtau))} = \vartheta(\bbtau)
	\end{align*}
\end{proof}

The following result shows how the transformed tail process describes what happens within a cluster.

\begin{proposition}
	\label{prop:gen-Z}
	Under the assumptions used to define the transformed anchored tail process and condition \(\mathcal{D}^{\prime} \), for every \(\tau >0\) and \(\bbtau^* = (\tau ,\dots \tau )\) and for \(\mathbb{X}_{n,1}\) as in~\eqref{eq:blocks} we have
	\[
		\mathcal{L} \left(\tilde{\pi }\left(\frac{1}{\tau }\mathbb{X}_{n,1}\right) \middle\vert \ \mathbb{X}_{n,1}\in \mathscr{W}^c _{r_n}(U(\bbtau^*)) \right) \to \mathcal{L} (\tilde{\pi }(Z))
	\]
	as \(n \to \infty \).
\end{proposition}

\begin{proof}
	Since \(\tilde{\mathscr{J}}\) is convergence determining, we prove that, for every \(\tilde{A} \in \tilde{\mathscr{J} }\) such that 
	\(\mu (\tilde{\pi } (\mathbf{Z} ) \in \partial\tilde{A}) = 0 \) we have 
	\[
		\mu (\mathbb{X}_{n,1} \in \tau \mathbb{A} \mid \mathbb{X}_{n,1} \in \mathscr{W}^c _{r_n} (U(\bbtau^*))) \to \mu ((\mathbf{Y}_j)_{j\in \mathbb{Z} } \in \mathbb{A} \mid \inf_{j \leq -1}  \min_{i = 1,\dots d} \abs{Y_{j,i}} >1)
	\]  

	We use the same estimate and decompose the event \(\mathbb{X}_{n,1} \in \mathscr{W}^c _{r_n} (U(\bbtau^*))\) according to the first exceedance by introducing 
	\[
		B_{n,i} \coloneqq  \mu (\mathbb{X}_{n,1} \in \tau \mathbb{A} \cap \mathscr{W} _i (U(\bbtau^*)) \cap \sigma^{-i}U(\bbtau^*))
	\] 
	and for \(i \geq q\) we define  
	\[
		D_{n,i} \coloneqq  \mu (\mathbb{X}_{n,1} \in \tau \mathbb{A} \cap \mathscr{W}_{i-q,i}(U(\bbtau^*)) \cap \sigma^{-i}U(\bbtau^*))
	\]
	so that we have 
	\[
		\abs{\mu (\mathbb{X}_{n,1} \in \tau \mathbb{A} \cap \mathscr{W} _{r_n} ^c (U(\bbtau^*))) - \sum_{i = q} ^{r_n -1} B_{n,i}} \leq q \mu (U_n(\bbtau^*)) \eqqcolon I_1(n,q)
	\]
	while
	\[
		\abs{\sum_{i= q} ^{r_n} B_{n,i} - \sum_{i=q} ^{r_n} D_{n,i}} \leq (r_n - q)\mu (U_n ^{(q)}(\bbtau^*) \cap \mathscr{W}_{[q+1, r_n)} ^c(U_n(\bbtau^*))) \eqqcolon I_2(n,q)
	\]
	using these two estimates and stationarity we can write 
	\[
		\left\lvert \mu \left(\frac{1}{\tau } \mathbb{X}_{n,1} \in \mathbb{A}\ \middle\vert \mathscr{W} ^c _{r_n}(U_n(\bbtau^*)) \right) - P_n \right\rvert \leq \frac{2I_1(q,n) + I_2(n,q)}{\mu (\mathbb{X} _{n,1} \in \mathscr{W} ^c _{r_n}(U_n(\bbtau^*)))} \eqqcolon E(n,q) 
	\]
	where 
	\[
		P_n = \frac{r_n \mu (U_n(\bbtau^*))}{\mu (\mathscr{W} ^c _{r_n}(U_n(\bbtau^*)))}\mu (\mathbb{X}_n ^{r_n - q_n, 2r_n}\in \tau \mathbb{A} \cap \mathscr{W}_{r_n -q} (U(\bbtau^*)) \mid \abs{\mathbf{X}_{r_n}} \nleq \mathbb{u}_n (\bbtau^*))
	\]
	First we need to show that the right-hand side goes to zero as \(n \to \infty\) (and then \(q \to \infty\)). By Corollary~\ref{SM:cor} we have that \(E(n,q) \leq C k_n (2I_1(q,n) + I_2(n,q))\) and the first term goes to 0 by definition of \(k_n\) while the second term goes by condition \(\mathcal{D}^{\prime} \). To conclude we need to show that \(P_n\) converges to \(\mu ((\mathbf{Y}_j)_{j\in \mathbb{Z} } \in \mathbb{A} \mid \inf_{j \leq -1}  \min_{i = 1,\dots d} \abs{Y_{j,i}} >1)\). Again, by Corollary~\ref{SM:cor} we have 
	\[
		\frac{r_n \mu (U_n(\bbtau^*))}{\mu (\mathscr{W} ^c _{r_n}(U_n(\bbtau^*)))} \to \vartheta(\bbtau^*)^{-1} 
	\]
	so that we are left to prove that \(\mu (\mathbb{X}_n ^{r_n - q_n, 2r_n}\in \tau \mathbb{A} \cap \mathscr{W}_{r_n -q} (U(\bbtau^*)) \mid \abs{\mathbf{X}_{r_n}} \nleq \mathbb{u}_n (\bbtau^*))\) converges to \(\mu ((\mathbf{Y}_j)_{j\in \mathbb{Z} } \in \mathbb{A} \mid \inf_{j \leq -1}  \min_{i = 1,\dots d} \abs{Y_{j,i}} >1)\). 
	This follows from the definition of \(\mathbf{Y} \) and the fact that the \(\set{(\mathbf{Y}_j)_{j \in \mathbb{Z} } \in \partial \mathscr{W}_{-\infty, 0} (U(\mathbf{1}))}\) is negligible.  
	Indeed, that set is contained in \(\bigcup_{j\leq 1} \bigcup_{i=1}^{d} \set{\abs{Y_{j,i}} = 1}\) and from the fact that \(\mathbf{Y}_j = \Theta_j \cdot \min_{i =1,\dots d} \abs{Y_{j,i}} \) and \(\Theta_j\) is independent of \(\min_{i =1,\dots d} \abs{Y_{j,i}} \) we have that \(\mu ( \bigcup_{i=1}^{d} \set{\abs{Y_{j,i}} = 1}) = 0\).     
\end{proof}

We can decompose the transformed anchored tail process \(\mathbf{Z} =  L_Z \mathbf{K}\) where
\[
	L_Z \coloneqq \norm{\mathcal{P}(\mathbf{Z} )}_0 ^{-1}  \quad \text{and} \quad \mathbf{K} = \frac{\mathbf{Z}}{L_Z}
\]
then it follows from the same arguments as in~\cite[Proposition~3.14]{freitas2025enrichedfunctionallimittheorems} that \(L_Z\) is uniformly distributed in \((0,1)\) and \(L_Z\) and \(\mathbf{K} \) are independent, which motivates the following decomposition map on \(\tilde{l}^d_\infty\): 
\begin{align}
	\label{eq:decomposition-Z}
	h \colon \tilde{l}^d_\infty \setminus \{\infty\} &\to \mathbb{R}_+ \times \mathbb{S}^d \notag\\
	\tilde{\mathbf{x}} &\mapsto \left( \frac{1}{\lVert \tilde{\mathcal{P}}(\tilde{\mathbf{x}})\rVert_0},\ \tilde{\mathbf{x}}\,\bigl\lVert \tilde{\mathcal{P}}(\tilde{\mathbf{x}})\bigr\rVert_0 \right)
\end{align}
where \(\mathbb{S} ^d\) is the unitary sphere in \(l_\infty ^d\).

\subsection{Convergence of clusters}

The next result gives more insight into how the tail process encodes the information
contained in a cluster, by establishing the convergence of the intensity of the cluster
points. Since the clusters lie in $\tilde{l}^d_\infty\setminus\{\boldsymbol\infty\}$, which is
not locally compact, we work with the notion of weak$^\#$ convergence on the space
$\mathcal{M}^\#_{\tilde{l}^d_\infty\setminus\{\boldsymbol\infty\}}$ of boundedly finite
measures. We refer to~\cite[Appendix A]{kulik2020heavy} or
to~\cite{freitas2025enrichedfunctionallimittheorems} for a precise definition of such
convergence.

\begin{proposition}
	\label{prop:cluster_conv}
Under the assumptions used to define the transformed anchored tail process, condition
$\CD^{\prime}$, and with $\mathbb{X}_{n,1}$ as in~\eqref{eq:blocks}, the sequence of
boundedly finite measures $k_n\mu(\tilde\pi(\mathbb{X}_{n,1})\in\cdot)$ converges in the
weak$^\#$ topology of $\mathcal{M}^\#_{\tilde{l}^d_\infty\setminus\{\boldsymbol\infty\}}$:
\begin{equation}
	\label{eq:cluster-intensity}
    k_n\mu\bigl(\tilde\pi(\mathbb{X}_{n,1})\in\cdot\bigr)
    \;\longrightarrow\;
    \Gamma(\bbtheta)\,(\mathrm{Leb}\times\mu_{\tilde{\mathbf K}})\circ h
    \;=:\;\eta(\cdot),
\end{equation}
where $\mu_{\tilde{\mathbf K}}$ is the distribution of $\tilde\pi(\mathbf K)$.
\end{proposition}

\begin{proof}
	Consider \(\tilde{A} \in \tilde{\mathscr{J} }  \) bounded, so that there exists \(\tau > 0\) such that \(\inf \norm{x_j} < \tau \) for every \(x \in \tilde{A} \). Hence if \(\tilde{\pi }(\mathbb{X}_{n,1}) \in \tilde{A} \) we have \(\mathbb{X} _{n,1} \in \mathscr{W}_{r_n} ^c(U(\bbtau^*))\).
	As before, set \(\bbtau^* = (\tau, \dots ,\tau)\)

	\begin{align*}
		 & \mu \left( \tilde{\pi }(\mathbb{X}_{n,1}) \in \tilde{A} \right) = \mu \left( \mathbb{X} _{n,1} \in \mathbb{A} \right) = \mu \left(\mathbb{X}_{n,1} \in \mathbb{A} \cap \mathscr{W}_{r_n} ^c(U(\bbtau^*)) \right) = \\
		 & = \mu \left(\frac{1}{\tau } \mathbb{X}_{n,1} \in \tau ^{-1} \mathbb{A} \mid \mathbb{X} _{n,1}\in \mathscr{W}_{r_n} ^c(U(\bbtau^*))\right)\mu \left( \mathbb{X} _{n,1}\in \mathscr{W}_{r_n} ^c(U(\bbtau^*)) \right)
	\end{align*}
	Hence, using that \(k_n\mu \left( \mathbb{X}_{n,1} \in \mathscr{W}_{r_n} ^c(U(\bbtau^*)) \right) \to \tau \Gamma(\bbtheta)\), we obtain
	\begin{align*}
		 & k_n \mu \left( \tilde{\pi }\left(\mathbb{X} _{n,1}\right)\in  \tilde{A} \right) \to \tau \Gamma (\bbtheta) \mu \left( \tau \mathbf{Z} \in \mathbb{A}  \right)
	\end{align*}
	And to conclude,  
	\begin{align*}
		&\mu (\tau \mathbf{Z} \in \mathbb{A}) = \int_0 ^1 \mu (\tau (\mathbf{Z}  _j)_j \in \mathbb{A} \mid L_ \mathbf{Z} = v)dv = \\ 
		&= \int_0 ^1 \mu (\tau v (\mathbf{K}  _j)_j \in \mathbb{A} \mid L_ \mathbf{Z} = v)dv
		= \frac{1}{\tau }\int_0 ^1 \mu (s (\mathbf{K}_j)_j \in \mathbb{A} )ds. 
	\end{align*}
\end{proof}
Before proving the convergence of the point processes, let us give some intuition as to why \(\Gamma(\bbtheta)\) arises naturally here. Unlike the univariate case, where the frequency is simply \(\theta\), one might at first expect the multivariate frequency to be \(\vartheta(\bbtau)\), yet this turns out not to be the case.
\begin{remark}
    It is not surprising that the frequency appears to be \(\Gamma(\bbtheta)\). Indeed, since \(\mathbf{Y} \) is obtained conditioning on having a big event controlled by common threshold  \(\tau\), and the (normalized) avoidance probability satisfies
    \[
        \mu\!\left( \mathbb{M} _n \leq \mathbb{u} _n(\bbtau^*) \right) \longrightarrow e^{-\Gamma(\bbtheta)\, \tau },
    \]
	then the void probabilities of the limiting point process were expected to be of a Poisson process of intensity \(\Gamma(\bbtheta)\). One thus expects the limit to be a marked Poisson process with (temporal) frequency \(\Gamma(\bbtheta)\), the marks being distributed as \(\mathbf{K} \), which by construction describes what happens inside a cluster.
\end{remark}

We define and prove the weak convergence of the point processes that keep all the cluster information. Consider a random element on \(\mathcal{N}^\# _{\mathbb{R}^+ _0 \times \tilde{l}^d _\infty \setminus \{\infty \} }\) of boundedly finite point measures by

\begin{equation}
	\label{eq:ppn}
	N_n = \sum_{i=1} ^\infty \delta_{(i/k_n, \tilde{\pi}(\mathbb{X}_{n,i}))}.
\end{equation}

Consider \((T_i, U_i)_{i \in \mathbb{N}} \) as points of a Poisson point process on \((\mathbb{R}^+ _0)^2\) with intensity \(\Gamma(\bbtheta) Leb^2\) and then consider \(N_\infty\) as the \((\mathbbm{1} , \tilde{\mathbf{K}} )\)-marking of that point process, i.e.
\begin{equation}
	\label{eq:ppf}
	N_\infty = \sum_{i=0} ^\infty  \delta_{(T_i, U_i \tilde{\mathbf{K}}_i)}
\end{equation}
where \(\tilde{\mathbf{K}}_i\) are i.i.d copies of \(\tilde{\mathbf{K}} \). This is a Poisson point process with intensity \(Leb \times \eta \), where \(\eta \) is the cluster measure defined above.\\

\noindent As in~\cite{freitas2020complete} we assume the existence of a piling measure \(\nu \) in \(l_\infty ^d\) such that
\[
	\nu (B_l) = \lim_{q \to \infty} \lim_{n \to \infty} n \mu \left( B_{n,l} ^{(q)}\right)
\]
for every \(B_l \in \mathscr{F} \). As a consequence
\begin{equation}
	\label{eq:equiv-measure}
	\nu (B_l) = \eta (\tilde{B}_{l})
\end{equation}
This is because, from corollary~\ref{SM:cor} and proposition~\ref{prop:cluster_conv} we have
\[
	\nu (B _l) = \lim_{q \to \infty}\lim_{n \to \infty} n \mu \left( B_{n,l} ^{(q)} \right) = \lim_{q \to \infty}\lim_{n \to \infty} k_n \mu \left( \mathbb{X}_{n,1} \in \mathbb{B}_{n,l} \right) = \eta (\tilde{B}_l)
\]

Now we can state the main result about convergence.

\begin{theorem}
	\label{teo:convergence}
	Let \(\mathbf{X}_0, \mathbf{X}_1,\dots \) be a stationary process of random vectors in \((\mathbb{R}^m)^d\) such that~\eqref{eq:freq-univ} holds. Assume that the transformed anchored tail process is well-defined and conditions \(\mathcal{D}\) and \(\mathcal{D}^{\prime}\) hold. Then point process \(N_n\), given in~\eqref{eq:ppn} converges weakly in \(\mathcal{N}^\# _{\mathbb{R} ^+ _0 \times \tilde{l}_\infty  ^d}\) to the Poisson process \(N_\infty \) given by~\eqref{eq:ppf}
\end{theorem}
We now present an example that serves as a direct test of a substantial portion of the results established in this work

\begin{example}
	\label{ex:tail}
    Consider a map $T \colon [0,1] \to [0,1]$ given by $T(x) = 2x \bmod 1$ 
    and the observables $\psi_i = -\log(\mathrm{dist}(x, \xi_i))$ for 
    $i = 1,2$, where $\xi_i$ are chosen so that $T(\xi_1) = \xi_2$ 
    and $\mathcal{O}^+(\xi_2) \cap \{\xi_1 , \xi_2\} = \emptyset$.

We compute the transformed anchored tail process for this example. 
First, since $\mathbf{Y}$ is independent of the choice of $\bbtau^*$, 
we write $U_n(\bbtau^*) = U_n$ and likewise for all sets derived from it. 
Given the form of the observables,
\[
    \set{\abs{\mathbf{X}_0} \nleq \mathbb{u}_n(\bbtau^*)} 
    = U_{n,1} \cup U_{n,2} 
    = \bigcup_{i=1}^2 B(\xi_i, e^{-u_{n,i}(\tau)}) 
    = \bigcup_{i=1}^2 B_{n,i},
\]
with
\[
    \frac{\mu(B_{n,i})}{\mu(U_n)} = \frac{1}{2} \quad \forall\, i = 1,2.
\]
To compute the process $(\mathbf{Z}_j)_{j \in \mathbb{Z}}$, we first compute 
the process $(\mathbf{Y}_j)_{j \in \mathbb{Z}}$ and then the event on which 
we condition, namely $\Set{\inf_{j \leq -1} \min_{i = 1,\dots d} \abs{Y_{j,i}} > 1}$, 
so that $\mathbf{Y}_0$ is the first observation above the threshold within 
the cluster.

Fixing \(s \leq 0,\ t>0\), we begin by computing the tail process $(\mathbf{Y}_j)_{j = s} ^t$ 
via its definition, that is, as the limit of
\[
    \left(
        \frac{1}{\tau}
        \left(
            \mathbb{u}_n^{-1}(\lvert \mathbf{X}_{r_n + s} \rvert), 
            \dots, 
            \mathbb{u}_n^{-1}(\lvert \mathbf{X}_{r_n + t} \rvert)
        \;\middle|\;
        \mathbf{X}_{r_n} \nleq \mathbb{u}_n(\bbtau^*)
        \right)
    \right).
\]
To this end we consider each case according to which set $T^{r_n}(x)$ 
belongs to. Suppose that $T^{r_n}(x) \in B_{n,1}$; then for $n$ sufficiently 
large we have $T^{r_n + j}(x) \notin U_n$ for $j \leq -1$, 
$T^{r_n+1}(x) \in B_{n,2}$, and $T^{r_n+j}(x) \notin U_n$ for $j \geq 2$. 
Moreover, the form of the observable gives
\[
    \mathrm{dist}(T^{r_n+1}(x), \xi_2) = 2\, \mathrm{dist}(T^{r_n}(x), \xi_1).
\]
Hence, in this case the process $(\mathbf{Y}_j)_{j \in \mathbb{Z}}$ is given by
\[
    W \begin{pmatrix}
        \dots & \infty & \overset{*}{1} & \infty & \infty & \dots \\
        \dots & \infty & \infty         & 2      & \infty & \dots
    \end{pmatrix},
\]
where $W \sim \mathcal{U}[0,1]$ by virtue of Lemma~\ref{lemma:uniformY0}, 
and where $*$ marks the position of $\mathbf{Y}_0$.

Turning to the case $T^{r_n}(x) \in B_{n,2}$, the key distinction comes from 
the decomposition
\[
    \{T^{r_n}(x) \in B_{n,2}\} 
    = \{T^{r_n}(x) \in B_{n,2} \cap T^{r_n-1}(x) \in B_{n,1}\} 
      \cup 
      \{T^{r_n}(x) \in B_{n,2} \cap T^{r_n-1}(x) \in B_{n,1}^c\},
\]
where
\[
    \frac{\mu(\{T^{r_n}(x) \in B_{n,2}\} \cap \{T^{r_n-1}(x) \in B_{n,1}\})}
         {\mu(B_{n,2})} = \frac{1}{2}.
\]
It follows that in this case the process $(\mathbf{Y}_j)_{j \in \mathbb{Z}}$ 
is given by
\begin{equation*}
    (\mathbf{Y}_j)_{j \in \mathbb{Z}} 
    = \mathbbm{1}_{K=1}\, W_1 
      \begin{pmatrix}
          \dots & \infty & \infty              & \infty & \dots \\
          \dots & \infty & \underset{*}{1}     & \infty & \dots
      \end{pmatrix}
    + \mathbbm{1}_{K=2}\, W_2 
      \begin{pmatrix}
          \dots & \tfrac{1}{2} & \infty          & \infty & \dots \\
          \dots & \infty       & \underset{*}{1} & \infty & \dots
      \end{pmatrix},
\end{equation*}
where $\mu(K=1) = \mu(K=2) = \tfrac{1}{2}$. Consequently, we can write 
the full process $(\mathbf{Y}_j)_{j \in \mathbb{Z}}$ as
\begin{multline*}
    (\mathbf{Y}_j)_{j \in \mathbb{Z}} 
    = \mathbbm{1}_{E=1}\, W_1
      \begin{pmatrix}
          \dots & \infty & \overset{*}{1} & \infty & \infty & \dots \\
          \dots & \infty & \infty         & 2      & \infty & \dots
      \end{pmatrix} \\
    + \mathbbm{1}_{E=2}\, W_2 
      \begin{pmatrix}
          \dots & \infty & \infty              & \infty & \dots \\
          \dots & \infty & \underset{*}{1}     & \infty & \dots
      \end{pmatrix}
    + \mathbbm{1}_{E=3}\, W_3 
      \begin{pmatrix}
          \dots & \tfrac{1}{2} & \infty          & \infty & \dots \\
          \dots & \infty       & \underset{*}{1} & \infty & \dots
      \end{pmatrix},
\end{multline*}
with $\mu(E=1) = \tfrac{1}{2}$ and $\mu(E=2) = \mu(E=3) = \tfrac{1}{4}$ and \(W_i \sim \mathcal{U}[0,1]\) for $i = 1, 2, 3$.

It remains to compute the probability of the conditioning event, namely 
$\inf_{j \leq -1} \min_{i = 1,\dots d} \abs{Y_{j,i}} > 1$:
\[
    \mu\!\left(\inf_{j \leq -1} \min_{i = 1,\dots d} \abs{Y_{j,i}} > 1\right) 
    = \mu(\{E=1\} \cup \{E=2\}) = \frac{3}{4}.
\]
The process $(\mathbf{Z}_j)_{j \in \mathbb{Z}}$ is therefore given by
\begin{equation*}
    (\mathbf{Z}_j)_{j \in \mathbb{Z}} 
    = \mathbbm{1}_{D=1}\, W_1 
      \begin{pmatrix}
          \dots & \infty & \overset{*}{1} & \infty & \infty & \dots \\
          \dots & \infty & \infty         & 2      & \infty & \dots
      \end{pmatrix} \\
    + \mathbbm{1}_{D=2}\, W_2 
      \begin{pmatrix}
          \dots & \infty & \infty              & \infty & \dots \\
          \dots & \infty & \underset{*}{1}     & \infty & \dots
      \end{pmatrix},
\end{equation*}
where
\[
    \mu(D=1) 
    = \frac{\mu(E=1)}{\mu\!\left(\inf_{j \leq -1} \min_{i = 1,\dots d} \abs{Y_{j,i}} > 1\right)} 
    = \frac{2}{3}, 
    \qquad 
    \mu(D=2) = \frac{1}{3}.
\]
\end{example}

\begin{example}[Separating systems via the anchored tail process]
\label{ex:dist-tail}
In~\cite{aimino2026multivariate} the authors computed the dependence structure via the stable tail dependence function $\Gamma(\bbtheta\bbtau)$ in various examples. In particular, in Examples 3.3.2 and 3.2.2 they obtain the \emph{same} stable dependence function, even though the underlying mechanism generating the rare events $(\mathcal{Z} _i)_{i =1,2}$ is different in the two systems. This shows that the stable tail dependence function alone need not distinguish between systems whose extremal behaviour is, in fact, genuinely different. Here we show that the tail process resolves precisely this ambiguity: computing $\mathbf{Z}$ for these two examples makes the difference explicit, disambiguating between the two systems whenever the tail process exists. In the following sections we will also see why it is natural for them to share the same dependence structure despite this distinction.
\medskip
\noindent\textit{Example 3.2.2:} The first example is a two periodic case for the doubling map, i.e.\
\[
T : [0,1] \to [0,1], \qquad T(x) = \lfloor 2x \rfloor,
\qquad Z_1 = \xi_1, \quad Z_2 = \xi_2
\]
s.t.\ $T(\xi_1) = \xi_2$ and $T^2(\xi_1) = \xi_1$. In this case the tail process is easy to compute (we refer to~\cite[Appendix D.1]{freitas2025enrichedfunctionallimittheorems} as it has been computed there the univariate version) and it is 
\[
\mathbf{Z}  = \mathbbm{1}_{E=0}\, W
\begin{pmatrix}
\overset{*}{1} & \infty & 4 & \infty & \cdots \\
\infty & 2 & \infty & 8 & \cdots
\end{pmatrix}
+ \mathbbm{1}_{E=1}\, W
\begin{pmatrix}
\infty & 2 & \infty & 8 & \cdots \\
\underset{*}{1} & \infty & 4 & \infty & \cdots
\end{pmatrix}
\]
with $W \sim \mathrm{Unif}(0,1)$ and $E \sim \mathrm{Bern}\!\left(\tfrac12\right)$.
\medskip

\noindent\textit{Example \ 3.3.2:} In this case we have the same map, but
$Z_1 = \{\xi_1, \xi_3\}$, $Z_2 = \{\xi_2, \xi_3\}$ s.t.\
$T(\xi_1) = T(\xi_2) = \xi_3$ and
$\theta^+(\xi_3) \cap \{\xi_1, \xi_2, \xi_3\} \neq \varnothing$.
Let us compute $\mathbf{Z} $ in this case: again, since $\mathbf{Y} $ (and $\mathbf{Z} $) are
independent of the choice of $\bbtau^*$, we write $U_n(\bbtau^*) = U_n$.
From
\[ \set{\abs{\mathbf{X}_{r_n}} \leq \mathbb{u}_n(\bbtau^*)} = U_{n,1} \cup U_{n,2}
= \bigcup_{i=1}^{3} B\!\left(\xi_i,\, e^{-u_{n,i}(\tau)}\right)
= \bigcup_{i=1}^{3} B_{i}
\]
with $\mu(B_{n,i})/\mu(U_n) = 1/3$. Since \(T^{-1} (\xi_3) = \set{\xi_1, \xi_2}\) we divide this set as 
\begin{align*}
	\set{T^{r_n}(x)  \in B_1} \cup \set{T^{r_n}(x) \in B_2} \cup \set{T^{r_n}(x) \in B_3, T^{r_n-1}(x) \cap B_1}\\
	\cup \set{T^{r_n}(x) \in B_3, T^{r_n-1}(x)  \cap B_2}
\end{align*}
with
\[
	\frac{\mu \left(\set{T^{r_n}(x) \in B_3, T^{r_n-1}(x) B_i}\right)}{\mu \left(\abs{\mathbf{X}_{r_n}} \leq \mathbb{u}_n(\bbtau^*)\right)} = \frac{1}{6}, \quad i = 1,2;
\]
and 
\[
	\frac{\mu \left(T^{r_n}(x) \in B_i\right)}{\mu \left(\abs{\mathbf{X}_{r_n}} \leq \mathbb{u}_n(\bbtau^*)\right)} = \frac{1}{3}, \quad i =1,2.   
\]
So tha tail process is given by 
\begin{multline*}
\mathbf{Y}  = \mathbbm{1}_{\hat E=0}\, W
\begin{pmatrix}
\overset{*}{1} & 2 & \infty & \cdots \\
\infty & 2 & \infty & \cdots
\end{pmatrix}
+ \mathbbm{1}_{ E=1}\, W
\begin{pmatrix}
\infty & 2 & \infty & \cdots \\
\underset{*}{1} & 2 & \infty & \cdots
\end{pmatrix} + \mathbbm{1}_{ E=3}\, W
\begin{pmatrix}
\frac{1}{2} & \overset{*}{1} & \infty & \cdots \\
\infty & 1 & \infty & \cdots
\end{pmatrix} \\ + \mathbbm{1}_{ E=4}\, W
\begin{pmatrix}
\infty & 1 & \infty & \cdots \\
\frac{1}{2} & \underset{*}{1} & \infty & \cdots
\end{pmatrix}
\end{multline*}
with \(W \sim \mathrm{Unif}(0,1) \) and \(\mu (E = 0) = \mu (E = 1) = \frac{1}{3}\), while \(\mu(E = 3) = \mu (E=4) = \frac{1}{6}\).    
The anchoring leads to the following anchored tail process \(\mathbf{Z} \) 
\[
\mathbf{Z}  = \mathbbm{1}_{\hat E=0}\, W
\begin{pmatrix}
\overset{*}{1} & 2 & \infty & \cdots \\
\infty & 2 & \infty & \cdots
\end{pmatrix}
+ \mathbbm{1}_{\hat E=1}\, W
\begin{pmatrix}
\infty & 2 & \infty & \cdots \\
\underset{*}{1} & 2 & \infty & \cdots
\end{pmatrix}
\]
with \(\hat E \sim \mathrm{Bin(\frac{1}{2})} \). 
So we showed that these two systems, despite having the same
stable tail dependence structure, can be distinguished via the
tail process.
\end{example}
\section{Spectral measure for stationary sequences}
\label{sec:spectral_measure}

In this section, within the framework of stationary sequences $(X_n)_n$ defined in~(2.1), we define the spectral measure associated with the stable tail dependence structure. The spectral measure is a central object in multivariate extreme value theory (see, e.g., \cite{beirlant2006statistics,segers2012max}), as it encodes the relative weights with which extreme events are shared among the components, and thereby characterizes their asymptotic dependence. The standard construction is based on a (pseudo-)polar decomposition with respect to two norms, $\|\cdot\|_1$ and $\|\cdot\|_2$, and the spectral measure is obtained as the angular part of the exponent measure; see~\cite{segers2012max} for details.

For stationary sequences, however, this construction must be adapted, because the components may exhibit temporal dependence. Large observations occurring at different times can still belong to the same extremal episode: for instance, an extreme value in one component may induce an extreme value in another after a finite number of time steps. We therefore group observations into clusters. This grouping is justified by the blocking argument: under the mixing condition introduced above, blocks of length $r_n$ become asymptotically independent by virtue of $D(u_n(\tau))$, while $D'(u_n(\tau))$ ensures that each block contains at most one cluster of extremes. We refer to~\cite[Theorem~2.6]{aimino2026multivariate} for an application of this blocking technique in the multivariate setting.

We first introduce the notion of an \emph{extremal profile}, namely a configuration describing the relative magnitudes of the components in an extreme observation. Let $X_n$ be as above, and let $u_n$ satisfy conditions~(2.2) and~(2.4). For each $n\in\mathbb{N}$ and $j\le n$ we define
\[
  \hat\rho_{n,j} := \min_{i=1,\dots,d} u_{n,i}^{-1}(|X_{j,i}|),
  \qquad
  \hat\omega_{n,j} := \frac{u_n^{-1}(X_j)}{\hat\rho_{n,j}} \in \Big\{x\in[1,\infty]^d : \inf_i x_i = 1\Big\} =: \Sigma_d .
\]
Here $\hat\rho_{n,j}$ measures the overall magnitude of the observation at time $j$, and $\hat\omega_{n,j}$ is the associated angular component, describing the relative magnitudes of the components. A small value of $\hat\rho_{n,j}$ means that a large observation occurred at time $j$; indeed,
\[
  \min_{i=1,\dots,d} u_{n,i}^{-1}(|X_{j,i}|) \le \tau \iff X_j \not\le u_n(\tau^*).
\]

The natural candidate for the spectral measure $H$ on $\Sigma_d$ is then the limit of
\[
n\,\mu\big(\hat\rho_{n,0}<1,\ \hat\omega_{n,0}\in\cdot\,\big)
\] 
as in the i.i.d. case. This choice, however, detects dependence between components only when they become extreme at the same time, which is the only mechanism of dependence available in the i.i.d. setting. For stationary sequences, components may also be dependent through extremes occurring at different times within the same cluster, and this is not seen by the construction above: two components may appear asymptotically independent merely because their large values are not simultaneous, even though they are linked by the dynamics. The extremal profiles of a stationary sequence must therefore be defined at the level of clusters. To this end, we group observations into blocks and keep track of the minimum in each component, which corresponds to recording the maximum observable per component, as in the definition of $M_n$. We therefore refine the pseudo-polar decomposition as follows. First, we define

\begin{equation*}
	m_n \coloneqq  \left(\bigwedge_{j = 0}^{r_n -1} u_{n,1}^{-1}(\abs{X_{j,1}}),\dots , \bigwedge_{j = 0}^{r_n - 1} u_{n,d}^{-1}(\abs{X_{j,d}})\right)
\end{equation*}

Let \((k_n)_n\) be given by condition \(\mathcal{D} ^{\prime} \), from Corollary~\ref{SM:cor} we have that \(\forall \bbtau \in [0, \infty ]^d \setminus \{\boldsymbol{\infty }\}\) we have   
\begin{equation}
	\label{eq:SM-kn}
	\lim_{n \to \infty} k_n\mu (m_n \ngeq \bbtau)  = \lim_{n \to \infty} k_n \mu\left(\mathscr{W} _{r_n} ^c (U_n(\bbtau)) \right) = \vartheta(\bbtau)\hat{\Gamma}(\bbtau) = \Gamma(\bbtheta\bbtau)
\end{equation}

Consider the space $E \coloneqq [0,\infty]^d \setminus \{\boldsymbol{\infty}\}$. It is an open
subset of the compact Polish space $[0,\infty]^d$, hence Polish; a compatible complete metric
is obtained as for the sequence space $l_\infty^d$ (see~\cite{kulik2020heavy}). Its relatively
compact sets are exactly those bounded away from $\boldsymbol{\infty}$. On $E$ we consider the
sequence of measures
\[
    \mu_n \coloneqq k_n\,\mu(m_n \in \cdot).
\]
We seek a limiting measure for this sequence. By~\cite[Proposition 3.16]{resnick1987extreme}, the sequence of boundedly finite measures $(\mu_n)_{n 0}$ is relatively compact for the vague topology if and only if
\[
    \sup_n \mu_n(B) < \infty \qquad \text{for every relatively compact } B \subseteq E.
\]
For each such $B$ there is a $\tau >0$ with $B \subseteq L_{\bbtau^*}$, where
\[
    L_{\bbtau ^*} \coloneqq \bigl\{x \in E : \min_{i=1,\dots,d} x_i < \tau \bigr\}
\]
Consequently, 
\[
    \mu_n(B) \le \mu_n(L_{\bbtau^*})
    = k_n\,\mu\bigl(\mathscr{W}^c_{r_n}(U_n(\bbtau^*))\bigr) < \infty
\]
by~\eqref{eq:SM-kn}, so $(\mu_n)_n$ is vaguely relatively compact.

To uniquely identify the limit, note that the sets $\{L_{\bbtau}\}_{\bbtau}$ form a
measure-determining class.
Hence there exists a unique boundedly finite measure $\Lambda$ on $E$ with $\mu_n \xrightarrow{\,v\,}
\Lambda$ along the full sequence, satisfying, for every $\bbtau \in E$,
\begin{equation}
\label{eq:SM_const}
    \Lambda(L_{\bbtau})
    = \Lambda\bigl(\{y \in E : y \not\ge \bbtau\}\bigr)
    = \Gamma(\bbtheta\,\bbtau).
\end{equation}
In particular $\Lambda\big|_{E} = \Lambda\big|_{(0,\infty]^d \setminus \{\boldsymbol{\infty}\}}$:
since $\Gamma$ is homogeneous of degree $1$, for every $\varepsilon>0$
\[
    \Lambda\bigl(\{y \in E : \min_i y_i = 0\}\bigr)
    \le \Lambda(L_{\varepsilon\mathbf{1}})
    = \Gamma(\varepsilon\,\bbtheta)
    = \varepsilon\,\Gamma(\bbtheta),	
\]
hence $\Lambda(\{y : \min_i y_i = 0\}) = 0$, which gives the claim.
Consider the following decomposition of \((0,\infty]^d \setminus \{\boldsymbol{\infty}\}\) into a radial and an angular part, given by the map \(\varphi\) defined as follows
\begin{align*}
\varphi \colon  (0,\infty ]^d \setminus & \{\boldsymbol{\infty}\} \rightarrow (0,\infty ) \times \Sigma_d \\
& y \mapsto \left(\min_{i=1,\dots, d} y_i,\ \frac{(y_1, \dots y_d)}{\min_{i=1,\dots, d} y_i}\right)
\end{align*}

And introduce the pseudo-polar decomposition by considering
\begin{equation}
	\label{eq:phi}
	\rho _n \coloneqq \varphi_1 (m_n) = \min_{i=1,\dots d} m_{n,i} = \norm{\mathcal{P}(\mathbb{X}_{n,1})}^{-1}, \quad \omega_n \coloneqq \varphi_2 (m_n) \in \Sigma_d 
\end{equation}
Note that while \(\rho_n \) encodes the largest observation inside the cluster, and in particular 
\[
	\rho_n < \tau \iff \mathbb{X}_{n,1} \in \mathscr{W} _{r_n} ^c (U_n(\bbtau^*)),
\] 
\(\omega_n\) retains only the largest observation per component within each cluster.  
As a result, the spectral measure we obtain encodes the dependence at leading order, i.e. the dominant asymptotic dependence between the components, while finer interactions are not captured.\\ 

We define the spectral measure \(H\) on \(\Sigma_d\) as the angular part of \(\Lambda\) 
\begin{equation}
	H(B) = \Lambda\left(y \in (0,\infty]^d \setminus \{\boldsymbol{\infty}\} \colon \varphi(y) < 1,\ \varphi_2(y) \in B\right), \qquad \forall B \subset \Sigma_d. 
\end{equation}

From properties~\eqref{eq:g1},~\eqref{eq:g2} of \(\Gamma\) and \(\vartheta(\cdot)\), we have that \(\forall a > 0\)  
\[ 
\Lambda (aL_\bbtau) = \Lambda(L_{a \bbtau}) = \vartheta(a\bbtau)\hat{\Gamma}(a \bbtau) = a\Lambda (L_\bbtau),
\]
i.e. \(\Lambda \) is 1-homogeneous, and therefore for all \(\tau > 0\) 
\[
	\Lambda\left(\varphi(y) < \tau ,\ \varphi_2(y) \in B\right) = \tau \Lambda \left(\varphi(y) < 1 ,\ \varphi_2(y) \in B\right) = \tau H(B).
\]
This yields us a decomposition of \(\Lambda\) as \( dr H(dw)\) in \((0,\infty)\times \Sigma_d\), i.e. for every \(\Lambda \)-integrable function \(g\)
\[
	\int g(y)\Lambda(dy) = \int_{\Sigma_d}\int_{0}^\infty g(rw) dr H(dw).
\]
Choosing, for a given \(\bbtau \in [0,\infty )^d\), \(g(y) = I_A(y)\) where \(A = \set{y \in (0, \infty]^d \setminus \{\boldsymbol{\infty}\}  : y \ngeq \bbtau}\), from~\eqref{eq:SM_const}, we obtain a useful representation of the stable tail dependence function that allows us to compute the spectral measure in concrete examples, indeed from~\eqref{eq:SM_const} and~\eqref{eq:SM-kn} the left-hand side becomes 
\begin{align*}
	\int_{(0, \infty]^d \setminus \{\boldsymbol{\infty}\}} \mathbbm{1}_{A}(y) \Lambda(dy) = \Lambda\left(y \in (0, \infty]^d \setminus \{\boldsymbol{\infty}\} \colon y \ngeq \bbtau \right) = \Lambda(L_\bbtau) = \Gamma(\bbtheta \bbtau) 
\end{align*}
while the right-hand side is 
\begin{align*}
	\int_{\Sigma_d } \int_0^\infty  \mathbbm{1}_{A}(r \omega) dr H(d\omega) = \int_{\Sigma_d} \int_0 ^\infty \mathbbm{1}_{(r,\omega ) \colon r < \max_{i =1,\dots d}\left(\frac{\tau_i}{\omega_i}\right)} dr H(d \omega) = \int_{\Sigma_d} \max_{i = 1,\dots d} \left(\frac{\tau _i}{\omega_i}\right) H(d \omega)
\end{align*}
This leads to a very useful representation and the stationary case analogue of the de Haan–Resnick representation for max-stable vectors, already obtained in~\cite{de1977limit}
\begin{equation}
\label{eq:spectral-representation}
	\Gamma(\bbtheta\bbtau) = \int_{\Sigma_d}\max_{i = 1, \dots, d}\left(\frac{\tau_i}{\omega_i}\right)H(d\omega)
\end{equation}

Since $\Lambda$ is obtained as the vague limit of $k_n\mu(m_n \in \cdot)$, 
for every bounded $\Lambda$-continuity set $A \in \mathcal{B}$ we have 
\[
    k_n\mu(m_n \in A) \to \Lambda(A).
\]
The set $A_B \coloneqq \{y \in (0,\infty]^d \setminus \{\boldsymbol{\infty}\} 
: \varphi(y) \in (0,1)\times B\}$ is bounded (it is contained in $L_1$), 
and by the product structure $\Lambda = dr\, H(d\omega)$,
\[
\Lambda(\partial A_B) 
    = \Lambda\bigl(\{r = 1\} \times B\bigr) + \Lambda\bigl((0,1) \times \partial B\bigr) = H(\partial B),
\]
hence for every $B \subset \Sigma_d$ with $H(\partial B) = 0$,
\begin{equation}
\label{eq:Hvague}
    H(B) = \Lambda(A_B) = \lim_{n\to\infty} k_n\mu(\rho_n < 1,\ \omega_n \in B),
\end{equation}
that is, $H$ is the weak\footnote{as the space is compact vague and weak convergence are equivalent} limit of the sequence of measures $k_n \mu(\rho_n < 1,\ \omega_n \in \cdot)$ on $\Sigma_d$.
Before elaborating on the properties of the spectral measure as defined here, we start by proving that \(H(\Sigma_d) = \Gamma(\bbtheta)\), which is exactly the frequency of the limiting point process \(N_\infty \) and that the extremal index \(\theta_i\) of each component can be obtained from \(H\).     

When we normalize the spectral measure by its total mass, we obtain the so-called \textit{profile  measure} 
\[
Q(\cdot) = \frac{H(\cdot)}{H(\Sigma_d)},
\]   

\begin{lemma}
	\label{lem:HG}
	\(H(\Sigma_d) = \Gamma(\bbtheta)\), where \(\bbtheta = (\theta_1 ,\dots \theta _d)\) and \(\theta_i\) is the extremal index of the \(i\)-th component of \((X_n)_{n \geq 1}\). Moreover, for every \(i \in \{1,\dots d\}\) we have  
	\begin{equation}
		\int_{\Sigma_d} \frac{H(d \omega)}{\omega_i} = \theta_i 
	\end{equation} 	
	where \(\frac{1}{\infty } \coloneqq  0\).  
\end{lemma}

\begin{proof}
	By definition of spectral measure
	\[
		\Gamma(\bbtheta \bbtau) = \int_{\Sigma_d}\max_{i = 1,\dots d}{\left(\frac{\tau_i}{\omega_i}\right)}H(d\omega)
	\]
	given that for every \(\omega \in \Sigma_d\) we have \(\omega \geq (1,\dots ,1)\) we obtain \(\max_{1,\dots d}\left(\frac{\tau_i}{\omega_i}\right) \leq \max (\tau_i)\).
	On the other hand, by definition of \(\Sigma_d\) we know that \(\min_i \omega_i = 1\), hence \(\min_{i=1,\dots d}(\tau_i) \leq \max_{i=1,\dots d}\left(\frac{\tau_i}{\omega_i}\right)\). Putting all together,
	\[
		\min_{i=1,\dots d} (\tau_i) \leq \max_{i=1,\dots d}\left(\frac{\tau_i}{\omega_i}\right) \leq \max_{i=1,\dots d} (\tau_i)
	\]
	Finally if we consider \(\tau_1 =\dots = \tau_d = 1\) we obtain
	\[
		\Gamma(\bbtheta) = \int_{\Sigma_d}H(d\omega) = H(\Sigma_d)
	\]
From property~\eqref{eq:g2} and~\eqref{eq:g_prop} we have  
\begin{equation*}
	\theta_i = \theta(e_i) = \theta (e_i) \hat{\Gamma}(e_i) = \Gamma(\bbtheta e_i) = \int_{\Sigma_d} \frac{H(d \omega)}{\omega_i}.
\end{equation*} 
\end{proof}

\begin{remark}
To the best of our knowledge, this result does not have a direct analogue in the existing literature.
In the i.i.d. setting (cf.~\cite{segers2012max,beirlant2006statistics}), the spectral measure satisfies constraints analogous to those in~\eqref{lem:HG}, but all components have the same value, determined by the choice of norm (typically \(1\)). In the stationary setting considered here, the corresponding imbalance is instead determined by the extremal indices of the individual components.
When all components share the same tail, the measure can be normalized and pushed forward to the simplex. In this case, our construction fully extends the classical spectral measure of~\cite{segers2012max,beirlant2006statistics} to the stationary setting.
\end{remark}
\subsection{Extracting the profile distribution from tail process}

In this section we examine the tail process more closely and 
establish its relationship with the spectral measure, or even better with the profile distribution \(Q\). Recall 
that, given the representation of the limiting point process 
$N_\infty$, the anchored tail process completely describes 
the behaviour inside each cluster, and therefore encodes all 
relevant information about the system. It is thus natural to 
expect that the spectral measure can be recovered as a 
projection of the tail process that retains only the 
essential information needed for multivariate extremal 
analysis, namely the dependence structure between components 
at the extremal level. This is the content of 
Corollary~\ref{cor:proj}.

We begin by exhibiting how the spectral measure arises in an 
explicit representation of the transformed tail process. 
Since $\mathbf{Z}$ is obtained as the limit of 
$\mathbb{X}_{n,1}$ conditioned on 
$\mathscr{W}^c_{r_n}(U_n(\bbtau^*))$, and since 
the spectral measure encodes different degrees of extremal 
dependence within that conditioning event, one expects 
$\mathbf{Z}$ to admit a mixture representation indexed by 
the extremal profiles. To make this precise we first verify 
that the relevant projection maps are continuous.

\begin{proposition}
The function $S \colon l_0^d \setminus \{\mathbf{0}\} \to 
\set{x \in \mathbb{R}_+ ^d \colon \max_{i=1,\dots,d} x_i = 1}$ defined by
\[
S(\mathbf{x}) \coloneqq \left(
\frac{\sup_{j \in \mathbb{Z}} |x_{j,1}|}{\norm{\mathbf{x}}_0}, \dots,
\frac{\sup_{j \in \mathbb{Z}} |x_{j,d}|}{\norm{\mathbf{x}}_0}
\right)
\]
is continuous.
\end{proposition}

\begin{proof}
For every $\mathbf{x} \in l_0^d \setminus \{\mathbf{0}\}$ 
we have $\max_{i=1,\dots,d} \sup_{j \in \mathbb{Z}} 
|x_{j,i}| > 0$, since otherwise $\mathbf{x} \equiv 0$; 
hence $S$ is well defined. 

We first show that the map 
$\mathbf{x} \mapsto \left(\sup_{j \in \mathbb{Z}} |x_{j,1}|, 
\dots, \sup_{j \in \mathbb{Z}} |x_{j,d}|\right)$ is 
continuous. If $\mathbf{x}_n \to \mathbf{x}_\infty$ in 
$l_0^d$, then $d(\mathbf{x}_n, \mathbf{x}_\infty) \to 0$ 
implies $\|\mathbf{x}_n - \mathbf{x}_\infty\|_0 \to 0$, 
i.e.
\[
\sup_{j \in \mathbb{Z}} \max_{i=1,\dots,d} 
|x_{n,j,i} - x_{\infty,j,i}| \to 0.
\]
The continuity of each component then follows from
\[
\left|\sup_j |x_{\infty,j,i}| - \sup_j |x_{n,j,i}|\right| 
\leq \sup_j \left| |x_{\infty,j,i}| - |x_{n,j,i}| \right|
\leq \sup_j |x_{\infty,j,i} - x_{n,j,i}|
\leq \|\mathbf{x}_n - \mathbf{x}_\infty\|_0.
\]
Since for every \(\mathbf{x}  \in l_0 ^d \setminus {\boldsymbol{0}}\) we have \(\norm{\mathbf{x}}_0 = \max_i \sup_j \abs{x_{j,i}} \), i.e. the norm of the sequence is exactly the sup-norm of the vector  \((\sup_{j\in \mathbb{Z} } \abs{x_{j,1}}, \dots, \sup_{j\in \mathbb{Z} } \abs{x_{j,d}})\) and the map $y \mapsto \mathbf{y} /\norm{\mathbf{y} }$ is continuous, the composition $S$ is continuous.
\end{proof}

Recalling~\eqref{eq:phi}, we can write 
$\omega_n = S_\infty(\mathbb{X}_{n,1})$, 
where $S_\infty \colon l_\infty^d \setminus \{\infty\} \to 
\Sigma_d $ is the composition
\[
l_\infty^d \setminus \{\infty\} 
\xrightarrow{\;\mathcal{P}\;} 
l_0^d \setminus \{\mathbf{0}\} 
\xrightarrow{\;S\;} 
\set{x \in \mathbb{R}_+ ^d \colon \max_{i=1,\dots,d} x_i = 1}
\xrightarrow{\;F\;} 
\Sigma_d,
\]
and $F \colon \mathbb{R}^d \setminus \{\mathbf{0}\} \to 
\overline{\mathbb{R}}^d \setminus \{\boldsymbol{\infty}\}$ is given by 
$F = (f, \dots, f)$ componentwise, with
\[
f(x) \coloneqq \begin{cases} 1/x & x \neq 0, \\ 
\infty & x = 0. \end{cases}
\]
Each factor in the composition is continuous, hence so is 
$S_\infty$.

\begin{corollary}\label{cor:proj}
Under the assumptions of Proposition~\ref{prop:gen-Z},
\begin{equation}\label{eq:SZSK}
\mathcal{L}(S_\infty(\mathbf{Z})) = 
\mathcal{L}(S_\infty(\mathbf{K})) = 
\mathcal{L}\!\left(\inf_j \abs{K_{j,1}}, \dots, \inf_j \abs{K_{j,d}}\right) = Q,
\end{equation}
and the law of $\mathbf{Z}$ disintegrates over $Q$ as
\begin{equation}\label{eq:Z-disint}
\mathcal{L}(\mathbf{Z}) = \int_{\Sigma_d} 
\mathcal{L}(\mathbf{Z} \mid S_\infty(\mathbf{Z}) \in 
d\omega)\, Q(d\omega) \eqqcolon \int_{\Sigma_d} 
\mathcal{L}(\mathbf{Z}^\omega)\, Q(d\omega).
\end{equation}
\end{corollary}

\begin{proof}
Since $S_\infty(\tilde{\pi}(\mathbf{x})) = S_\infty(\mathbf{x})$ for every $\mathbf{x} \in l_\infty^d$, 
the continuous mapping theorem together with Proposition~\ref{prop:gen-Z} gives that, for every \(B \subset \Sigma_d \) such that \(\mu (S_\infty (\mathbf{Z} ) \in \partial B) = 0\),
\[
\lim_{n \to \infty} \mu\!\left(S_\infty(\mathbb{X}_{n,1}) \in B \mid 
\mathscr{W}^c_{r_n}(U_n(\bbtau^*))\right) = \mu(S_\infty(\mathbf{Z}) \in B).
\]
On the other hand, from the definition of the spectral measure in Section~\ref{sec:spectral_measure} for every $B \subset \Sigma_d$ with $Q(\partial B) = 0$,
\[
\lim_{n \to \infty} \mu\!\left(S_\infty(\mathbb{X}_{n,1}) \in B \mid 
\mathscr{W}^c_{r_n}(U_n(\bbtau^*))\right) 
= \lim_{n \to \infty} 
\frac{\mu\!\left(\omega_n \in B,\, 
\mathscr{W}^c_{r_n}(U_n(\bbtau^*))\right)}
{\mu\!\left(\mathscr{W}^c_{r_n}(U_n(\bbtau^*))\right)} = Q(B).
\]
From uniqueness of the weak limit, we conclude that \ $Q = \mathcal{L}(S_\infty(\mathbf{Z})) $.
The other equality comes from the decomposition $\mathbf{Z} = (L_\mathbf{Z}, \mathbf{K})$ 
as in~\eqref{eq:decomposition-Z}. The disintegration~\eqref{eq:Z-disint} is then the 
standard disintegration theorem for probability measures on Polish spaces.
\end{proof}
\begin{example}
    We continue the example considered above, in which we computed the transformed anchored tail process. According to Corollary~\ref{cor:proj}, the spectral measure can be obtained as the projection of the tail process, so we can already say that the profile distribution will be
    \[
        Q(d\omega) = \frac{2}{3}\,\delta_{(1,2)} + \frac{1}{3}\,\delta_{(\infty,1)}.
    \]
    We check this by using the representation~\eqref{eq:spectral-representation} 
    after computing $\Gamma(\bbtheta\,\bbtau)$.

    First, we recall that $U_n(\bbtau)$ consists, for $n$ sufficiently large, 
    of disjoint sets $B_j(\tau_j)$, $j = 1,2$. Hence,
    \[
        \hat{\Gamma}(\bbtau) 
        = \lim_{n \to \infty} n\,\mu(\mathbb{X}_0 \nleq \mathbb{u}_n(\bbtau)) 
        = \tau_1 + \tau_2.
    \]
    We now compute $\vartheta(\bbtau)$ via formula~\eqref{eq: eimv}. 
    Since $\mathcal{O}^+(\zeta_2)$ does not contain any point at which
    the observables are maximal, we have, again for $n$ sufficiently large,
    \[
        A^{(q)}_n = A^{(1)}_n \quad \text{for all } q \geq 1.
    \]
    Now, for every $\alpha \in [0,1]$ 
    we define\footnote{We are using the homogeneity property of the 
    Lebesgue measure.}
    \[
        \gamma_1(\alpha) 
        \coloneqq \lim_{n \to \infty} \frac{\mu\!\left(B_1(\alpha) \setminus f^{-1}(B_2(1-\alpha))\right)}
                       {\mu(B_1(\alpha ))}, 
        \qquad 
        \gamma_2(\alpha) 
        \coloneqq \lim_{n \to \infty} \frac{\mu(B_2(1-\alpha))}{\mu(B_2(\alpha ))},
    \]
    so that, setting $\alpha = \dfrac{\tau_1}{\tau_1 + \tau_2}$,
    \[
        \vartheta(\bbtau) 
        \coloneqq \alpha\,\gamma_1(\alpha) + (1-\alpha)\,\gamma_2(\alpha) 
        = \max\!\left(\alpha - \frac{1-\alpha}{2},\, 0\right) + (1-\alpha).
    \]
    This can be rewritten as
    \[
        \vartheta(\bbtau) 
        = \max\!\left(\alpha,\, \frac{1-\alpha}{2}\right) 
          + \frac{1}{2}(1-\alpha).
    \]
    Since 
    $\Gamma(\bbtheta\,\bbtau) 
     = \vartheta(\bbtau)\,\hat{\Gamma}(\bbtau) 
     = \max\!\left(\tau_1,\,\dfrac{\tau_2}{2}\right) + \dfrac{1}{2}\tau_2$, 
    it follows immediately that
    \[
        H = \delta_{(1,2)} + \frac{1}{2}\,\delta_{(\infty,1)},
    \]
    from which the representation of $Q$ announced at the beginning of the 
    example follows.
\end{example}

Now we can rephrase the convergence of the cluster measure in view of Corollary~\ref{cor:proj}

\begin{proposition}
	Under the assumption used to define the transformed anchored tail process and condition \(\mathcal{D}^{\prime} \), we have
	\begin{equation}
		k_n \mu \left( \tilde{\pi }(\mathbb{X}_{n,1}) \in \cdot  \right) \to \Gamma(\bbtheta)(Leb\times \mu_{\tilde{\textbf{K}} })\circ h\eqqcolon \eta (\cdot) = \int_{\Sigma_d}  \int_0 ^\infty \mathbb{E} \left(\mathbbm{1}_{(\cdot)} s \mathbf{K} ^\omega \right)ds H(d \omega)
	\end{equation}
	in \(\omega^\#(\mathcal{M}^\# _{l_\infty ^d \setminus{\infty }})\) topology, where \(\mu_{{\tilde{\textbf{K} }}}\) is the distribution of  \(\tilde{\pi} \left(\mathbf{K} \right)\)
\end{proposition}

\begin{proof}
	From Proposition~\ref{prop:cluster_conv} we just need to show the last equality, but that is a direct consequece of Corollary~\ref{cor:proj}. 
\end{proof}

The following lemma is a direct consequence of the preceding theorem and will enable us to recover the limiting distribution of the componentwise maximum, as stated in~\cite{aimino2026multivariate}. 

Then from theorem~\ref{teo:convergence} we have 

\[
\lim_{n \to \infty} \mu\left( \mathbb{M}_n \leq \mathbb{u}_n (\bbtau)\right) = \lim_{n \to \infty}  \mu \left( N_n([0,1]\times U(\bbtau)) = 0 \right) =  e^{-\eta (\tilde{U}(\bbtau))}
\] 
hence we just need to compute the measure of \(\hat{U}(\bbtau)\), which is given by the cluster measure. 

\begin{lemma}
	\label{lemma:avoidance}
	Under the assumption used to define the transformed anchored tail process,  for every \(\bbtau = (\tau _1,\dots ,\tau_d) \in (0,\infty )^d\) and \(U (\bbtau) \) above, we have
	\[
		k_n \mu \left( \tilde{\pi }(\mathbb{X} _{n,1}) \in \tilde{U}(\bbtau) \right) \to \Gamma(\bbtheta\bbtau).
	\]
\end{lemma}

\begin{proof}
	From Proposition~\ref{prop:cluster_conv}, we just need to show that
	\[
			\int_{\Sigma_d} \int_0 ^{\max_i \tau_i} \mu \left( s\mathbf{K} ^\omega \in \mathbb{U} (\bbtau)\right)ds H(d\omega) = \Gamma(\bbtheta \bbtau)
	\]
	Now, for a fixed \(\omega \in \Sigma_d\) we have 
	\begin{align*}
		s \mathbf{K}^\omega \in \mathbb{U}(\bbtau) \iff \exists j \in \mathbb{Z} \colon s  \mathbf{K}^\omega _j  \in U(\bbtau) \iff s \in \bigcup_{i=1}^{d} \bigcup_{j \in \mathbb{Z}} \left\{ s \lvert \pi_i(\mathbf{K} _j^\omega) \rvert < \tau_i \right\} = \bigcup_{j \in \mathbb{Z}} \bigcup_{i=1}^{d} \left\{ s \lvert \pi_i(\mathbf{K} _j^\omega) \rvert < \tau_i \right\} 
	\end{align*}
	where we can freely exchange the unions, since \(i\) ranges in a finite set.  
	Hence,  
	\begin{align*}
		& \int_{\Sigma_d}  \int_0 ^{\max_i \tau_i} \mu \left( s\mathbf{K} ^\omega \in \mathbb{U} (\bbtau)\right)ds H(d \omega) = \int_{\Sigma_d} \int_0 ^{\max_i \tau_i}\mu \left( \bigcup_{j \in \mathbb{Z}} \bigcup_{i=1}^{d} \left\{ s \lvert \pi_i(\mathbf{K} _j^\omega) \rvert < \tau_i \right\} \right)ds H(d \omega) = \\
		&\int_{\Sigma_d}  \int_0 ^{\max_i \tau_i}\mu \left( s < \max\left(\frac{\tau_i}{\inf_{j\in \mathbb{Z} } \abs{\pi_i(\mathbf{K}^\omega _j)}}\right)\right) ds H(d\omega)=
	 	\int_ {\Sigma_d}  \max\left(\frac{\tau_i}{\inf_{j\in \mathbb{Z} } \abs{\pi_i(\mathbf{K}^\omega _j)}}\right) H(d\omega)= \\
		& = \int_{\Sigma_d} \max\left(\frac{\tau_i}{\omega_{i}}\right) H(d\omega) = \Gamma(\bbtheta \bbtau)
	\end{align*}
	where the second to last equality holds from Corollary~\ref{cor:proj}

\end{proof}

Let us take a step back and try to understand the consequences of Corollary~\ref{cor:proj}. Consider the following diagram 
\[
\begin{tikzcd}
{\Gamma(\bbtheta \bbtau)} & {} & H &&& {} & {(\mathbf{Z}_j)_{j\in \mathbb{Z}}} \\
&&&& {Q} \\
\\
&&&&&& {(\mathbf{Y})_{j\in \mathbb{Z} }}
\arrow["{\text{Eq.}~\eqref{eq:spectral-representation}}", Rightarrow, 2tail reversed, from=1-1, to=1-3]
\arrow[Rightarrow, from=1-3, to=2-5]
\arrow["{\text{Corollary}~\ref{cor:proj}}", equals, 2tail reversed, from=2-5, to=1-7]
\arrow[Rightarrow, from=4-7, to=1-7]
\arrow[Rightarrow, from=4-7, to=2-5]
\arrow["{\text{Proposition~\ref{prop: tetaY} and~\ref{prop:gammaY}}} "
  shift left,
  curve={height=60pt},
  equals, 2tail reversed,
  from=1-1,
  to=4-7,
]
\end{tikzcd}
\]
The diagram above summarizes the relationships between the objects introduced so far.
In general, it is not possible to reconstruct the full tail process \(\mathbf{Y} \) or \(\mathbf{Z} \) from the spectral measure \(H\)  alone. Corollary~\ref{cor:proj} confirms that the profile measure, though a useful practical tool, is a coarser object than the full process \(\mathbf{Z} \). By encoding only the most relevant dependence information between variables, it offers significant analytical and computational simplifications. The following example illustrates a setting in which the spectral measure can be computed explicitly, while the full tail process appears difficult, and possibly impossible 
to determine in closed form. This motivates the use of the spectral measure as a more tractable object 
that still captures the essential extremal dependence structure of the system.

\begin{example}

Let $\mathcal{X} = [0,1]^2$ and $T(x,y) = (2x \bmod 1, 2y \bmod 1)$, the standard dyadic map on the unit square. The invariant measure is two-dimensional Lebesgue measure $\mu$, and the system has summable decay of correlations against \(L^1\), so that $\mathcal{D} $ and $\mathcal{D} ^{\prime} $ are 
satisfied. We consider two observables maximised respectively on
\[
\mathcal{Z}_1 = \Delta 
\qquad \text{and} \qquad
\mathcal{Z}_2 = \Delta \cap \left[\tfrac{1}{3}, \tfrac{2}{3}\right]^2
\qquad \text{and},
\]
where $\Delta$ denotes the diagonal of $\mathcal X$. The corresponding exceedance regions are
rectangles\footnote{More precisely, rectangles in the local coordinates given by $\Delta$ and its
orthogonal complement, up to a negligible term of order $o(n^{-2})$.} with base $\mathcal Z_i$ and width $h_i$ chosen so that 
$n\mu(U_{n,i}(\tau_i)) \to  \tau_i$, as in~\eqref{eq:freq-univ}, giving $h_{1,n} = \tau_1/n$ and $h_{2,n} = 3\tau_2/n$.

\medskip
Since $\Delta$ is invariant under $T$, the relevant one-dimensional dynamics is governed by $f \colon [0,1] \to [0,1]$, $f(x) = 2x \bmod 1$. Setting $\Delta_0 \coloneqq [\frac{1}{3}, \frac{2}{3}]$, one checks that $f(\Delta_0) = \Delta_0^c$. For each $j \geq 1$ we denote by $R_j \subset [0, \frac{1}{3}]$ and 
$R_j^* \subset [\frac{2}{3}, 1]$ the two connected components of $f^{-j}(\Delta_0)$ such that
\[
f(R_j) = R_{j-1}, \quad f(R_j^*) = R_{j-1}^*, \quad
f^j(R_j) = f^n(R_j^*) = \Delta_0 \qquad (j \geq 1).
\]
These sets generate a sequence of refining partitions $(\mathcal{P}_m)_{m \geq 0}$ of $\Delta_0$ based on the successive returns to the set. Denote
$\mathcal{P}_0 = \{\Delta_0\}$ and, for $m \geq 1$,
\[
\mathcal{P}_m = \{\mathcal{I}_w \mid w \in \mathbb{N}^m\},
\qquad
\mathcal{I}_{w_1,\dots,w_m} \coloneqq \Delta_0 \cap 
\bigcap_{i=1}^{n} f^{-\left(i + \sum_{j=1}^{i-1}w_j\right)}
(R_{w_i} \cup R_{w_i}^*).
\]
The index $w = (w_1, \dots, w_m)$ records the successive excursion lengths outside $\Delta_0$: a point 
$x_0 \in \mathcal{I}_{w_1,\dots,w_m}$ spends $w_i$ steps in $\Delta_0^c$ at its $i$-th excursion, and first returns to $\Delta_0$ at time $m + \sum_{i=1}^m w_i$. For instance, 
$x_0 \in \mathcal{I}_{1,1}$ follows the itinerary
\[
x_0 \in \Delta_0 \;\to\; f(x_0) \in R_1 \cup R_1^* 
\;\to\; f^2(x_0) \in \Delta_0 \;\to\; 
f^3(x_0) \in R_1 \cup R_1^* \;\to\; f^4(x_0) \in \Delta_0.
\]
A direct computation gives $|\mathcal{I}_w| = \prod_{i=1}^n 2^{-w_i} |\Delta_0|$, which is consistent 
with $\sum_{w \in \mathbb{N}^n} |\mathcal{I}_w| = |\Delta_0|$.

\begin{figure}[H]
    \centering
    \includegraphics[width=0.6\textwidth]{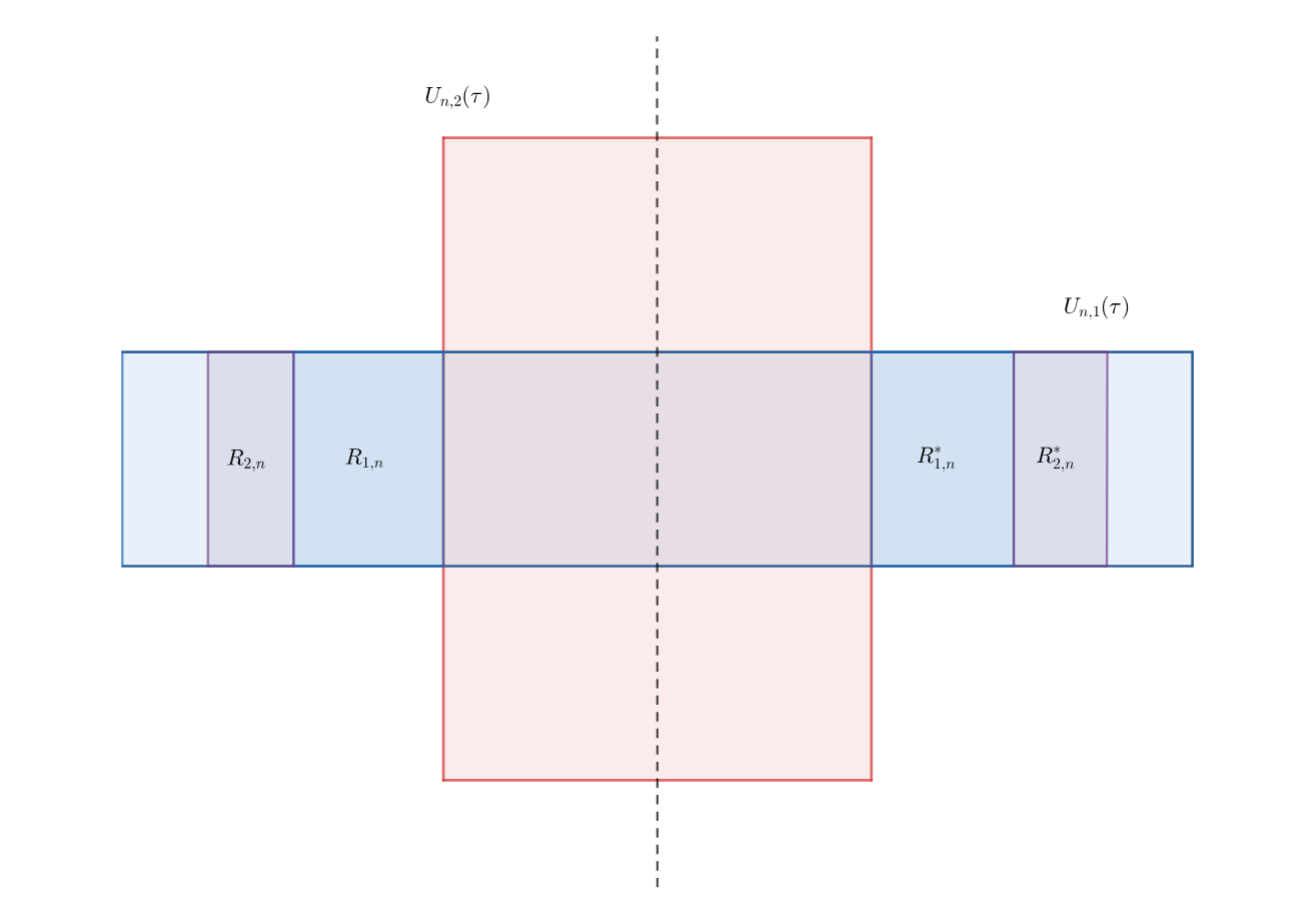}
    \caption{A scheme of the extremal sets \(U_{n,1}(\tau_1)\), \(U_{n,2}(\tau_2)\) and the sets \(R_{j,n}, R_{j,n} ^*\), when \(\tau_2 = \tau_1\)}
\end{figure}

\medskip
For every \(m \geq 1\), the partition $\mathcal{P}_m$ naturally induces a partition on $U_{n,2}(\tau_2)$ denoted as $\mathcal{V}_m$, and similarly the sets \((R_j ,R^* _j)_{j \geq 1}\) induce a partition on \(U_{n,1}(\tau_1) \cap \left[1/3, 2/3\right]^c \times [0,1]\). Explicitly, under this local coordinates
\[
\mathcal{V}_m = \left\{ \mathcal{I}_w \times \left(-\tfrac{h_{2,n}}{2}, \tfrac{h_{2,n}}{2}\right) \;:\; w \in \mathbb{N}^m \right\},
\]
and the sets $(R_j, R_j^*)_{j\geq1}$ induce, in the same way, the partition
\[
\left\{R_{j,n} \coloneqq R_j \times \left(-\tfrac{h_{1,n}}{2}, \tfrac{h_{1,n}}{2}\right),\; R _{j,n} ^* \coloneqq  R_j^* \times \left(-\tfrac{h_{1,n}}{2}, \tfrac{h_{1,n}}{2}\right) \;:\; j \geq 1 \right\}.
\] For $T^{r_n}(x) \in \mathcal I_w \times (-h_{2,n}/2, h_{2,n}/2)$ with $w=(w_1,\dots,w_m)$, 
the tail process can be described explicitly up to time $N \coloneqq m + \sum_{i=1}^n w_i$. 
The tail process has the form
\[
(\mathbf{Y}_j)_{j =0} ^N = W \begin{pmatrix} 1 \\ \frac{1}{3} \end{pmatrix} \odot 
\left(\overset{*}{B_1},\, B_2,\, \dots,\, B_m\right),
\]
where $W \sim \mathcal{U}[0,1]$. The $j$-th block $B_j$ 
corresponds to the $j$-th excursion of the orbit along $\Delta$, starting close to $\mathcal{Z}_2$, spending $w_j - 1$ 
steps on $\Delta \cap \left( \left[\frac{1}{3},\frac{2}{3}\right]^2 \right)^c$, and returning to $\Delta \cap \left[\frac{1}{3},\frac{2}{3}\right]^2$. Setting $s_j \coloneqq \sum_{i=1}^{j-1} w_i + (j-1)$ 
for the time at which the $j$-th excursion begins, it reads
\[
B_j = 2^{s_j}
\begin{pmatrix} 
1 & 2 & \cdots & 2^{w_j} \\ 
1 & \infty & \cdots & \infty 
\end{pmatrix}.
\]

For example, for $T^{r_n}(x) \in \mathcal{I}_{1,1}$ one has 
$s_1 = 0$, $w_1 = w_2 = 1$, $s_2 = 2$, giving
\[
(\mathbf{Y} _j)_{j=0}^N = W \begin{pmatrix} 1 \\ \frac{1}{3} \end{pmatrix} \odot
\begin{pmatrix}
\infty & \overset{*}{1} & 2 & 4 & 8  \\
\infty & 1 & \infty & 4 & \infty 
\end{pmatrix}
\]
where $*$ marks the time of first entrance into $U_n(\bbtau)$. The tail process is known 
explicitly only up to time $m + \sum_{i=1}^n w_i$; since $|\mathcal{I}_w|$ decays exponentially in 
$\sum_i w_i$, a closed-form expression for the full tail process appears impossible to compute.

\begin{figure}[H]
    \centering
    \includegraphics[width=1\textwidth]{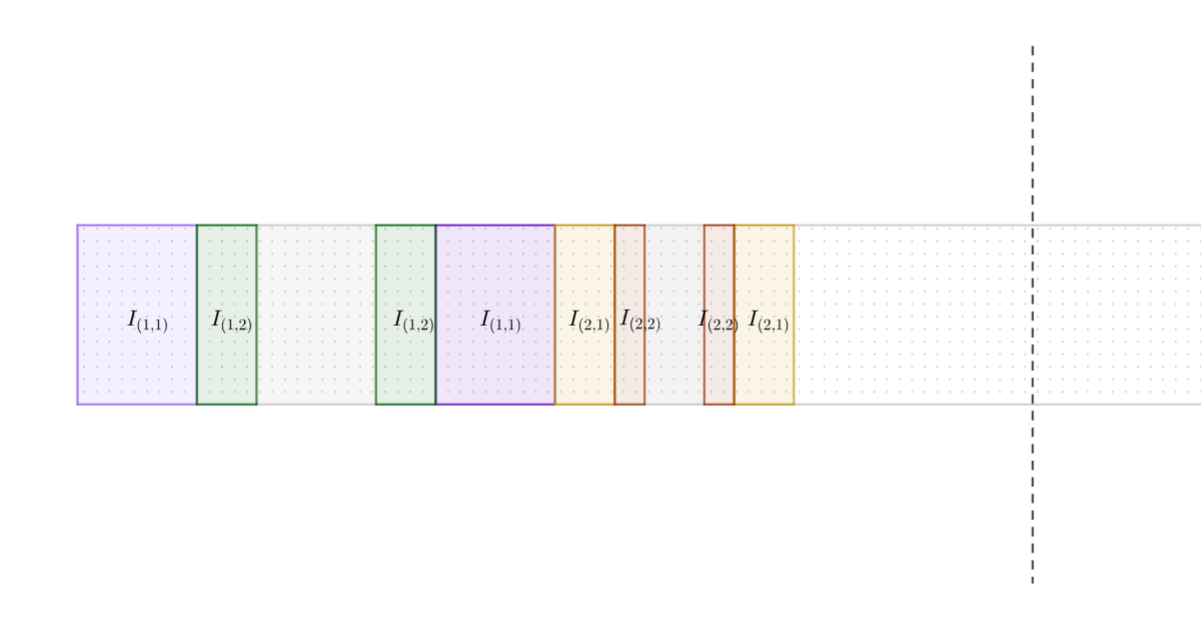}
    \caption{A sketch of the nested partition \(\mathcal{V}_2\) for \(U_n(\bbtau)\cap \left[\frac{1}{3}, \frac{2}{3}\right]^2\). The partition is symmetric with respect to \(1/2\)}
    \label{fig:}
\end{figure}

\medskip Despite the complexity of the tail process, the spectral 
measure is computable. Nonetheless, we show how we can compute the spectral measure and the stable tail dependence function for this system. 
We start from this first simple observation : since our system is expanding the dependence structure will only depend on a finite number of iterates of the system, in particular the minimum number of steps that are needed to visit both extremal sets. 
 Since $T$ is expanding, the extremal dependence structure is determined by the minimum number of iterates needed to visit both $\mathcal{Z}_1$ and 
$\mathcal{Z}_2$. According to where the first exceedance occurs, the componentwise minima of the cluster frequencies are as follows:
\noindent So it is not hard to see that according to where \(x_0\) (the first exceedance) lies, we obtain the following componentwise mimima for the frequencies 

\renewcommand{\arraystretch}{1.5}
\begin{table}[h!]
\centering
\begin{tabular}{c|c}

Location of the first exceedance  & Componentwise mimimum frequency in  \(\Sigma_d \)  \\ 
\hline
\(\Delta_0\) & \((3,1)\)  \\ 
\(R_1 \cup R^* _1\)  & \((\frac{3}{2},1)\) \\
\((R_n \cup R^* _n)_{n \geq 2}\) & \((1, \frac{2^n}{3})\) \\ 

\end{tabular}
\caption{Correspondence between \(x_0\) and component-wise mimima for the frequencies.}
\label{tab:example}
\end{table}
\noindent Since the spectral measure is supported on these componentwise minima, it takes the atomic form
\begin{equation}\label{eq:smes}
H = K_0\,\delta_{(3,1)} + K_1\,\delta_{(3/2,\,1)} 
+ \sum_{n \geq 2} K_n\,\delta_{(1,\,2^n/3)},
\end{equation}
where the weights are determined inductively below.

In order to compute \(\Gamma(\bbtheta \bbtau)\) we decompose the computation according to whether the
first exceedance lands in \(\Delta_0 \times [0,1]\) or in \(\mathcal{R}_{j,n} \coloneqq  R_{j,n} \cup R_{j,n} ^*\) for some \(j \geq 1\).  
\begin{enumerate}[label = (\roman*)]
\item In the first case \(x_0 \in U_{n,1}(\tau_1) \cap (\Delta_0 \times [0,1])^c\), in this case there is an \(n\in \mathbb{N} \) such that  \(x_0 \in U_{n,1}(\tau_1 ) \cap \mathcal{R}_n \) define the asymptotic entry probability
\[
	\gamma_n(\bbtau) \coloneqq  \frac{\mu (U_n (\bbtau) \cap \mathcal{R}_n)}{\mu (U_n(\bbtau))} 
\]
and the conditional probability of not producing a further large observation as  
\[
	\theta_n(\bbtau) \coloneqq  \frac{\mu (U_n(\bbtau) \cap \mathcal{R}_n \setminus \bigcup_{i=1}^{n-1} T^{-j} (\mathcal{R}_{n-j}) \cup T^{-n}(U_n(\bbtau)))}{\mu (U_n(\bbtau) \cap \mathcal{R}_n )}
\] 
\item Similarly, for \(x_0 \in U_n(\bbtau ) \cap (\Delta_0 \times [0,1])\)we define 
\[
	\hat{\gamma}_n(\bbtau) \coloneqq  \frac{\mu (U_n(\bbtau) \cap \Delta_0 \cap T^{-1} (\mathcal{R}_n) )}{\mu (U_n(\bbtau))}
\]
and \(\hat{\theta}_n(\bbtau)\)  
\[
	\hat{\theta}_n(\bbtau) \coloneqq \frac{\mu (U_n(\bbtau) \cap \Delta_0 \cap T^{-1} (\mathcal{R}_n)) \setminus \bigcup_{i=1}^{n} T^{-j} \mathcal{R}_{n-j+1} \cup T^{-(n+1)} (U_n(\bbtau) \cap \Delta _0)}{\mu (U_n(\bbtau) \cap \Delta_0)}.
\] 
\end{enumerate}

Now putting all together, \(\Gamma (\bbtheta \bbtau) = \sum_{i = 1}^\infty \gamma_n(\bbtau) \theta_n(\bbtau) + \hat{\gamma}_n(\bbtau)\hat{\theta }_n(\bbtau) \) is given by the following expansion 
\begin{multline*}
    \Gamma(\bbtheta \bbtau) = \frac{1}{3}\sum_{k=1}^\infty \left(\frac{1}{2}\right)^k \max \left(0, \left(\max (\tau_1 , 3\tau _2) - \max \left(\frac{\tau _1}{2}, \frac{3 \tau_2}{ 2^{k+1}}\right)\right)\right) + \\
    +  \frac{2}{3}\sum_{k=1}^\infty \left(\frac{1}{2}\right)^k \max \left(0, \tau_1 - \max \left(\frac{\tau _1}{2}, \frac{3 \tau _2}{2^k}\right) \right) 
\end{multline*}

\noindent To compute the spectral measure we consider a partition of $\Sigma_d$ using the sets
\[
    S_k = \Set{\frac{3}{2^{k}} \tau_2 < \tau_1 \leq \frac{3}{2^{k-1}} \tau_2}, \quad k \geq 1, \quad S_0 = \set{3\tau _2 < \tau_1}
\]
so that for $\bbtau \in S_k$ we have a closed formula for $\Gamma(\bbtheta\bbtau)$:
\[
    \Gamma(\bbtheta\bbtau)\big|_{S_k}
    = (\tau_1 + \tau_2)
    \left[
        \frac{5}{6} - \frac{1}{3 \cdot 4^k}
        + \frac{\tau_1}{\tau_1 + \tau_2}
        \left(
            \frac{1}{3 \cdot 2^k} - \frac{5}{6} + \frac{1}{3 \cdot 4^k}
        \right)
    \right].
\]

The atoms of the spectral measure can be determined inductively. On each region \(S_k\), the coefficients multiplying \(\tau_2\) depend solely on \(K_0, \dots K_k\). Consequently, the \(n\)-th atom is computable from its predecessors, and the first atom \(K_0 = \frac{1}{2}\), \(K_1 = \frac{1}{4}\) and \(K_n = \frac{1}{3} \frac{1}{2^n}\) for \(n \geq 2\).       
\end{example} 

\section{Systems covered by the theory}
\label{sec:systems}
The results of the previous sections apply to any measure-preserving dynamical system
$T\colon\mathcal{X}\to\mathcal{X}$ whose associated process $(\mathbf{X}_n)_{n\in\mathbb{N}}$,
defined as in~\eqref{eq:sequence-n}, admits a well-defined transformed anchored tail process
and satisfies conditions $\mathcal{D}$ and $\mathcal{D}^{\prime}$: by
Theorem~\ref{teo:convergence}, these are all that is needed for the point
processes to converge. This section identifies concrete classes of systems meeting these
requirements. Both have been verified for a wide
range of dynamical systems.
Condition $\mathcal{D}$ follows from sufficiently fast decay of correlations, with summable rates being enough, and can therefore be verified globally for a large class of systems. Condition $\mathcal{D}'$ is typically more delicate. Decay of correlations against $L^1$ provides a convenient global criterion for it. Without such a strong form of mixing, $\mathcal{D}'$ can still be established in principle, but this usually requires analysing the recurrence of the extreme set.

Consequently, the main examples to which the theory applies are expanding maps, such as:
\begin{itemize}
    \item uniformly expanding maps of the circle or the interval;
    \item Markov maps;
    \item piecewise expanding maps with countably many branches, such as Rychlik maps;
    \item higher-dimensional uniformly expanding maps.
\end{itemize}
For all of these, summable decay of correlations against $L^1(\mu)$ holds, and both $\mathcal{D}$ and $\mathcal{D}'$ can be derived from it.

The framework also covers some systems beyond the uniformly expanding setting. For instance, it was recently shown in~\cite{kirsebom2026visitingtimestatistics} that quadratic maps $T(x) = ax(1-x)$ with a Benedicks-Carleson parameter enjoy exponential mixing for intervals against $L^1$, and thus fit the present setting.Moreover, systems with weaker mixing, such as the LSV map, can be treated whenever their first return map belongs to one of the classes above: one works with the induced map, building a tail process that incorporates the inducing time, and then lifts the conclusions back to the original system.

\subsection{Extending the Spectral Measure for Induced Maps}
\label{sec:spectral-induced}

The aim of this section is to adapt the spectral measure construction
to non-uniformly expanding maps whose induced map satisfies
conditions~$\mathcal{D}$,~$\mathcal{D}'$.  The key novelty with
respect to~\cite{freitas2025enrichedfunctionallimittheorems} is the
presence of $d\ge 2$ observables and the consequent need to work in a
two-level sequence space that records both the \emph{inter-cluster}
structure (at the level of the induced map~$\overline{T}$) and the
\emph{intra-cluster} structure (what happens inside each induced
period).

Let $T\colon\mathcal{X}\to\mathcal{X}$ be a dynamical system on a
Lebesgue space $(\mathcal{X},\mathcal{B},\mu)$.  Fix a base set
$\Delta_0$ with $\mu(\Delta_0)>0$, a countable partition
$\mathcal{P}$ (mod~$\mu$) of~$\Delta_0$, and a return function
$R\colon\Delta_0\to\mathbb{N}$ that is constant on each element
of~$\mathcal{P}$.  We denote by $\overline{T}$ the induced map on
$\Delta_0$, and by $\mu_0$ the restriction of $\mu$ to $\Delta_0$.

Fix $d\ge 2$ observables $\psi_i\colon\mathcal{X}\to\mathbb{R}$ and
write $\bbpsi=(\psi_1,\dots,\psi_d)$.  Define the for each component $i$
\begin{equation}
  \label{eq:induced-max}
  \overline{\psi}_i(x)
  \;\coloneqq\;
  \max_{j=0,\dots,R(x)-1}\lvert\psi_i(T^j(x))\rvert,
\end{equation}
recording the largest value seen in component~$i$ during the entire
induced period starting at~$x$.

The construction of the space of decorating sequence requires two nested levels of sequence spaces,
but adapted to the $L^\infty$ norm on $\mathbb{R}^d$. 

To account for the two-level blocking,
we introduce the space
\begin{equation}
  \label{eq:ltilde-star}
  \widetilde{l}^*_\infty
  \;\coloneqq\;
  \Bigl\{
    \mathbf{x}=(\tilde{\mathbf{x}}_j)_{j\in\mathbb{Z}}
    \in(\tilde{l}_\infty^d)^{\mathbb{Z} }
    \colon
    \lim_{|j|\to\infty}\norm{\tilde{x}_j}_\infty = 0
  \Bigr\},
\end{equation}
where $\norm{\tilde{x}_j}_\infty \coloneqq \norm{\widetilde{\mathcal{P}}(\tilde{\mathbf{x} })}_0$
is the norm on $\tilde{l}_\infty^d$ induced by $\mathcal{P}$.  Its
quotient $\widetilde{l}^*_\infty/{\sim}$ under the shift on
index is the space in which the transformed anchored tail process for
the induced system takes values. 

Assume there exist normalising sequences $(u_{n,i}(\tau_i))_{n\in\mathbb{N}}$
such that
\begin{equation}
  \label{eq:normalisation}
  n\,\mu\bigl(\overline{\psi}_i > u_{n,i}(\tau_i)\bigr)
  \;\longrightarrow\; \tau_i,
  \qquad\forall\,i=1,\dots,d.
\end{equation}
Write $\mathbb{u}_n=(u_{n,1},\dots,u_{n,d})$ and define the
normalised observation
\begin{equation}
  \label{eq:Fn}
  \mathbb{F}_n(x)
  \;\coloneqq\;
  \mathbb{u}_n^{-1}(\bbpsi(x))\odot\operatorname{sign}(\bbpsi(x))
  \;\in\;\mathbb{R}^d,
\end{equation}

Consider the sequence $(\mathbf{X}_j)_{j\ge 0}$ given by
$\mathbf{X}_j=\bbpsi(\overline{T}^j(x))$.  For each
induced step $j$, the contribution of the entire induced excursion
starting at $\overline{T}^j(x)$ is encoded as the element
\begin{equation}
  \label{eq:Xtilde}
  \widetilde{\mathbf{X}}_j
  \;\coloneqq\;
  \tilde{\pi}\!\left(
    \mathbb{F}_n\!\left(\overline{T}^j(x)\right),\,
    \mathbb{F}_n\!\left(T\!\left(\overline{T}^j(x)\right)\right),\,
    \dots,\,
    \mathbb{F}_n\!\left(T^{R(\overline{T}^j(x))-1}\!\left(\overline{T}^j(x)\right)\right)
  \right)
  \;\in\;\tilde{l}_\infty^d,
\end{equation}
where $\tilde{\pi}$ is the natural projection from $l_\infty^d$ to
$\tilde{l}_\infty^d$.  

In particular, for the definition of the transformed  anchored tail process, we assume the existence of a process \((\tilde{\mathbf{Y} }_j)_{j \in \mathbb{Z} } \in l_\infty ^*\) satisfying 

\[
	      \mathcal{L}\left(\frac{1}{\tau} \left(\tilde{\mathbf{X}} _{r_n+s}, \dots,  \tilde{\mathbf{X}} _{r_n+t}\right)\ \middle\vert \ \norm{\mathcal{P} (\tilde{\mathbf{X}}_{r_n})}_0 > \tau^{-1} \right) \to \mathcal{L}((\tilde{\mathbf{Y}} _j)_{j = s,\dots t})
\]
for all \(\tau>0, s<t, s,t \in \mathbb{Z}\). 

The law of the transformed anchored tail process for the induced system is then defined as 
\[
	\mathcal{L} (\tilde{\mathbf{Z}}_j) = \mathcal{L} \left(\tilde{\mathbf{Y}}_j\ \mid \ \inf_{j \leq -1} \norm{\mathcal{P} (\tilde{\mathbf{Y}}_j)}_\infty > 1 \right)
\]
and its spectral decomposition is given by 
\[
	L_ \mathbf{Z} = (\sup_j \norm{\mathcal{P} (\mathbf{Z}_j)} )^{-1}, \qquad  \tilde{\mathbf{K}} = \frac{\tilde{\mathbf{Z} }}{L_ \mathbf{Z}} 
\]

With this definition we can prove that, if the induced map satisfies conditions $\mathcal{D}$ and $\mathcal{D}'$, then the point process of exceedances for the induced system converges to a Poisson point process with intensity measure given by the tail measure of the transformed anchored tail process for the induced system (~\cite[Proposition 4.8]{freitas2025enrichedfunctionallimittheorems}), and moreover once we define the spectral measure for the induced system, then we recover~\ref{cor:proj} so that it is still possible to recover it as the projection of the tail measure of the transformed anchored tail process for the induced system. Since we need to perform a two level blocking, then the projection will be 
\[
  Q
  \;=\;
  \mathcal{L} \left(\frac{\min_{j, k \in \mathbb{Z} }\pi_1(\widetilde{\mathbf{K} }_{j,k}), \dots, \min_{j, k \in \mathbb{Z} }\pi_d(\widetilde{\mathbf{K} }_{j,k})}
       {\min_{i = 1,\dots d} \min_{j,k}\pi_{i}(\widetilde{\mathbf{K} }_{j,k})}\right),
\]

Let $S_j\coloneqq\sum_{i=1}^{j}R(\overline{T}^i(x))$ denote the
cumulative return time up to induced step~$j$ (with $S_0=0$).  For
each $j=1,\dots,r_n$ define the vector of highest normalized
observations across each component within the $j$-th induced period, which equivalently corresponds to the componentwise minimum of the $j$-th block of frequencies 
\begin{equation}
  \label{eq:mn-j}
  (m_n)_j
  \;\coloneqq\;
  \left(
    \bigwedge_{s=S_{j-1}+1}^{S_j} u_{n,1}^{-1}(|X_{s,1}|),
    \;\dots,\;
    \bigwedge_{s=S_{j-1}+1}^{S_j} u_{n,d}^{-1}(|X_{s,d}|)
  \right)
  \;\in\;\mathbb{R}^d.
\end{equation} 

The radial variable for the entire block of $r_n$ induced observations
is defined by two successive minimization, first over components
(within each induction), then over induced periods (within the block):
\begin{equation}
  \label{eq:rho-n}
  \rho_{n,j}
  \;\coloneqq\;
  \min_{k=1,\dots,d}(m_n)_{j,k},
  \qquad
  \rho_n
  \;\coloneqq\;
  \min_{j=1,\dots,r_n}\rho_{n,j}.
\end{equation}
The event $\{\rho_n<1\}$ thus means that at least one component
exceeds the threshold in at least one of the $r_n$ induced periods,
which is the correct induced analogue of the uninduced exceedance event.
The angular variable is then
\begin{equation}
  \label{eq:omega-n}
  \omega_{n,k}
  \;\coloneqq\;
  \frac{\displaystyle\min_{j=1,\dots,r_n}(m_n)_{j,k}}{\rho_n},
  \quad k=1,\dots,d,
  \qquad
  \omega_n=(\omega_{n,1},\dots,\omega_{n,d})\in\Sigma_d.
\end{equation}

Under conditions $\mathcal{D}$, $\mathcal{D}'$ for the induced
map~$\overline{T}$, we can prove that 
\begin{equation}
  \lim_{n\to\infty}k_n\,\mu\!\left(\mathbf{X}_j\in
    \mathscr{W}_{r_n}^c(U_n(\mathbf{\bbtau}))\right)
  \;\to\;
  \overline{\vartheta}(\bbtau)\,\hat{\Gamma}(\bbtau) = \Gamma (\overline{\bbtheta} \bbtau),
\end{equation}
where $\overline{\vartheta}(\bbtau)$ is the extremal Index of the
induced system. 
Then we can define the spectral measure for the induced sysis the Extremal tem as the limit of the rescaled measure of the exceedance event with a given angular component as before
\begin{equation}
  \label{eq:spectral-measure}
  H(B)
  \;\coloneqq\;
  \lim_{n\to\infty}
  k_n\,\mu\!\bigl\{\rho_n < 1,\;\omega_n\in B\bigr\},
  \qquad B\subseteq\Sigma_d\text{ Borel.}
\end{equation}
The total mass satisfies
\begin{equation}
  H(\Sigma_d)
  \;=\;
  \lim_{n\to\infty}k_n\,\mu(\rho_n<1)
  \;=\;
  \lim_{n\to\infty}k_n\,\mu\!\left(\mathbf{X}_j\in
    \mathscr{W}_{r_n}^c(U_n(\mathbf{1}))\right)
  \;\longrightarrow\;
  \Gamma (\overline{\bbtheta}),
\end{equation}

\begin{example}
Consider the LSV map $T_\gamma$ 
\begin{equation}
    T_\gamma(x) \coloneqq
    \begin{cases}
        x(1 + 2^\gamma x^\gamma) & \text{if } x \in [0, 1/2), \\
        2x - 1 & \text{if } x \in [1/2, 1].
    \end{cases}
\end{equation}
with $\gamma\in(0,1)$, equipped with
its absolutely continuous invariant measure $\mu_\gamma$.  Take the
canonical inducing domain $\Delta_0=(1/2,1]$ and denote by $R$ the
first return time to $\Delta_0$.  Let $\{\zeta_1,\zeta_2\}$ be a
period-2 orbit of $T_\gamma$ with $\zeta_1\in\Delta_0$ and
$\zeta_2=T_\gamma(\zeta_1)\notin\Delta_0$. The induced map $\overline{T}_\gamma$ is a Rychlik map with
exponential decay of correlations, so conditions $\mathcal{D}$,
$\mathcal{D}^{\prime} $ hold.

Consider the two observables
\begin{equation}
  \phi_1(x)=|x-\zeta_1|^{-1/\alpha},
  \qquad
  \phi_2(x)=|x-\zeta_2|^{-1/\alpha},
  \qquad 0<\alpha<1.
\end{equation}

Consider \(\bbtau^* = (\tau ,\tau )\), then since 
\[
	\frac{\mu (U_{n,i}(\tau))}{\mu (U_n(\tau ))} = \frac{1}{2}
\] 
then the process \((\tilde{\mathbf{K}}_j)_{j \in \mathbb{Z} }\) is given by  
\begin{equation*}
    (\tilde{\mathbf{K}}_j)_{j \in \mathbb{Z}} = (\dots \infty,\ \infty,\ \tilde{K}_0(E),\ \tilde{K}_1(E),\ \tilde{K}_2(E), \dots )
\end{equation*}
where
\begin{align*}
    \tilde{\mathbf{K}}_0(1) &= \begin{pmatrix}
        \ldots & \infty & \infty & 1         & \infty    & \infty & \ldots \\
        \ldots & \infty & \infty & \infty    & \chi_1    & \infty & \ldots
    \end{pmatrix} \\[10pt]
    \tilde{\mathbf{K}}_j(1) &= \begin{pmatrix}
        \ldots & \infty & \infty & \infty              & \chi_2^j\chi_1^{j+1} & \infty & \ldots \\
        \ldots & \infty & \infty & \chi_2^j\chi_1^j    & \infty               & \infty & \ldots
    \end{pmatrix} \\[10pt]
    \tilde{\mathbf{K}}_0(2) &= \begin{pmatrix}
        \ldots & \infty & \infty & \infty & \infty & \ldots \\
        \ldots & \infty & 1      & \infty & \infty & \ldots
    \end{pmatrix} \\[10pt]
    \tilde{\mathbf{K}}_j(2) &= \begin{pmatrix}
        \ldots & \infty & \chi_2^j\chi_1^{j-1} & \infty              & \infty & \ldots \\
        \ldots & \infty & \infty               & \chi_2^j\chi_1^j    & \infty & \ldots
    \end{pmatrix}
\end{align*}
and \(E \sim \mathrm{Ber}(1/2) \) 
with $\chi_1 = DT_\gamma(\zeta_1)$ and $\chi_{2} = DT_\gamma(\zeta_2)$. The profile distribution is then given by 
\[
	Q = \frac{1}{2} \delta_{(1,\chi_1)} + \frac{1}{2}\delta_{(\chi_2,1)}
\]
\end{example}
\section{Consequences of the point process convergence}

The limiting REPP \(N_\infty \)  encodes a wealth of information, which can be extracted by projecting or restricting it to specific regions. We first show that \(N_\infty \)  recovers the classical univariate limiting point process obtained in an earlier work~\cite{freitas2025enrichedfunctionallimittheorems}. We then turn to a more genuinely multivariate viewpoint, deriving a limit for the multivariate exceedances point process introduced in~\cite{nandagopalan1994multivariate}, which keeps record not only of the number of exceedances but also which components contribute to each extreme observation. 

\subsection{Univariate point processes}
Once the convergence of the multivariate point process has been established, the classical univariate limit processes can be recovered by restricting the cluster measure to suitable subsets of the state space. For each component \(i = 1,\dots d\), consider the canonical projection
\[
	\Pi_i \colon l^d _{\infty } \to l_\infty
\]
on the \textit{i-th} component, which is continuous, hence by the continuous mapping theorem, the image under \(\Pi_i \) of \(N_\infty  \) satisfies  
\[
	(\mathbb{1} \times \Pi_i)(N_n) \rightarrow (\mathbb{1} \times \Pi_i)N_\infty 
\]
in \(\mathcal{N} ^\# _{\mathbb{R}^+ _0 \times l_\infty \setminus \{\infty \}}\) 
However, it is not immediately clear from this projection alone whether the resulting process coincides with the one obtained directly in the univariate setting. We must still verify that the associated cluster measure converges to the univariate limit given in~\cite{freitas2025enrichedfunctionallimittheorems}. To this end we introduce the family of sets which are used to check convergence of the cluster measure in the univariate case. These set are of the form \(A \in \mathscr{F} ^u \) where 
\[
\mathscr{F} ^u \coloneqq  \set{\set{\mathbf{x} \in l_\infty \colon x_j \in H_j, j=1,\dots N} \colon H_j \in \mathscr{F}_\mathcal{V},\ j=1,\dots N, N\in \mathbb{N}} 
\]
and \(\mathscr{F}_\mathcal{V} \) is the field generated by rectangles in \(\mathbb{R}^m \). 
We know that the (univariate) limiting process is 
\[
	\sum_{l=1}^{\infty} \delta_{\left( T_l,U_l \mathbf{K}^i _{l,j}\right)};
\]
where \(\left( T_l ,U_l \right)\) are points of a PRM of parameter \(\theta_i\), \(\mathbf{K}^i _{l,j}\) are i.i.d copies of the angular component of the (univariate) \textit{tail anchored transformed process} and the univariate cluster measure 
is given by 
\[
	\eta_i = \theta_i (\mathrm{Leb} \times \mu_{\mathbf{Q}})\circ h^i 
\] 
where we denoted with \(h^i\) the corresponding decomposition map, defined similarly to~\eqref{eq:decomposition-Z} but with \(\mathbf{K}^i\) instead of \(\mathbf{K}\). 
Finally, we want to show that for every \(A \in \mathscr{F} ^u\) we have
\[
	\lim_{n \to \infty} k_n \mu (\tilde{\pi} (\mathbb{X}_{n,1}) \in \Pi_i ^{-1} (\tilde{A})) = \theta_i \left(\mathrm{Leb} \times \mu_{\mathbf{\mathbf{K} }^i}\right) \circ h^i (\tilde{A}) \eqqcolon \eta_i (\tilde{A}), 
\]
for all bounded \(\tilde{A}\) such that \(\eta_i (\partial \tilde{A}) = 0\).
\begin{proposition}
	Under the assumptions used to define the transformed anchored tail process and condition \(\mathcal{D}^{\prime}\), we have for every \(i = 1, \dots d\)  
	\[
		((\Pi_i)_*\eta)(\tilde{A}) = \Gamma(\bbtheta)((\Pi_i)_*(\mathrm{Leb}\times \mu_{\textbf{K}})\circ h)(\tilde{A}) = \theta_i (\mathrm{Leb}  \times \mu _{\mathbf{K}^i })\circ h^i (\tilde{A}) = \eta_i(\tilde{A}) 
	\]
	where \(\mathbf{K}^i \) is the angular component of the \textit{i-th} univariate transformed anchored tail process.   
\end{proposition}
\begin{proof}
	Consider \(\tilde{A} \in \mathscr{F}^u \) bounded and such that \(\eta_i(\partial\tilde{A}) = 0\), since it is bounded there exists a \(\tau > 0\) such that for every \(\mathbf{x}  \in \tilde{A}\) we have \(\inf_{j \in \mathbb{Z} }\lvert x_j \rvert < \tau\). Hence, if \(\Pi_i (\tilde{\pi}(\mathbb{X}_{n,1})) \in \tilde{A}\) then \(\Pi_i(\mathbb{X}_{n,1}) \in \mathscr{W}^c _{r_n}(U_{n,i}(\tau) ) \), so writing \(\mathbb{A} = \tilde{\pi}^{-1} (\tilde{A})\) we have 
	\begin{align*}
		&\left((\Pi_i)_* \eta \right)(\tilde{A}) = \lim_{n \to \infty} k_n\mu \left(\Pi_i (\tilde{\pi }(\mathbb{X}_{n,1})) \in \tilde{A}\right) = \\
		&\lim_{n \to \infty}k_n\mu \left(\frac{\Pi_i(\mathbb{X}_{n,1})}{\tau } \in \tau^{-1}   \mathbb{A}, \Pi_i(\mathbb{X}_{n,1}) \in \mathscr{W} ^c _{r_n}(U_{n,i}(\tau))\right) = \\
		& \lim_{n  \to \infty} k_n \mu \left(\frac{\Pi_i(\mathbb{X}_{n,1})}{\tau } \in \tau^{-1}  \mathbb{A} \mid \Pi_i(\mathbb{X}_{n,1}) \in \mathscr{W} ^c _{r_n}(U_{n,i}(\tau))\right) \frac{\mu \left(\mathscr{W} ^c _{r_n}(U_{n,i}(\tau)) \right)}{\mu (\abs{X_{0,i}}  > u_{n,i}(\tau ) )} \mu (\abs{X_{0,i}} > u_{n,i}(\tau))  
	\end{align*} 
	the term \(k_n \mu \left(\frac{\Pi_i(\mathbb{X}_{n,1})}{\tau } \in \tau ^{-1} \mathbb{A} \mid \Pi_i(\mathbb{X}_{n,1}) \in \mathscr{W} ^c _{r_n}(U_n(\tau))\right)\) converges to \(\frac{\eta_i(\tilde{A})}{\theta_i \tau }\) from~\cite[corollary 3.16]{freitas2025enrichedfunctionallimittheorems}, and since the second term converges to \(\theta_i \tau \), the conclusion follows.         
\end{proof}
\begin{remark}
	Since Corollary~\ref{cor:proj} gives us an explicit connection between  \(\mathbf{Z} \) and the spectral measure, we can say more about \(\eta_i\). In particular if we define two elements \((q_j)_{j \in \mathbb{Z} }\) and \((k_j)_{j\in \mathbb{Z} }\) in \(l_\infty\) to be  \textit{angular-equivalent} if there is \(n \in \mathbb{Z} \) such that 
    \[
        \frac{1}{\inf(\abs{q_j})}\sigma^n(\mathbf{q} ) = \frac{1}{\inf (\abs{k_j})}(\mathbf{k} ),
    \]
	then it is easy to see that 
	\begin{align*}
        & ((\Pi_i)_*\eta)(\tilde{A}) = \Gamma (\bbtheta) \int_{\Sigma_d} \int_0 ^\tau \mu \left( s \textbf{K}^\omega \in \Pi_i ^{-1} (\mathbb{A}) \right)ds Q(d \omega) =\\  
        & \int_{\Sigma_d} \int_0 ^{\omega_{i}\tau }\mu \left( s\frac{\Pi_i(\mathbf{K} ^\omega)}{\omega_i} \in \mathbb{A}\right) ds \frac{H(d \omega)}{\omega_{i}}
    \end{align*}
	The first thing to notice is that 
    \[
    \int_0 ^{\omega_{i}\tau }\mu \left( s\frac{\Pi_i(\mathbf{K} ^\omega)}{\omega_i} \in \mathbb{A}\right) ds= \int_0 ^{\tau }\mu \left( s\frac{\Pi_i(\mathbf{K} ^\omega)}{\omega_i} \in \mathbb{A}\right) ds
    \]
    indeed \(\omega_{i} \geq 1\) for all \(\omega \in \Sigma_d\) and given that  \(\inf_j \abs{\Pi_i (\mathbf{K}^\omega)} = \omega _i \), when \(s > \tau \), \( s \Pi_i(\textbf{K}^m)\)  cannot be in \(\tilde{A}\).            
    Hence, we can write 
    \[
        ((\Pi_i)_*\eta)(\tilde{A}) = \theta_i \int_{\Sigma_d} \int_0 ^\tau \mu\left( s \frac{\Pi_i(\mathbf{K} ^\omega)}{\omega_i} \in \mathbb{A}  \right)ds \frac{H(d \omega)}{\theta_i \omega_i} 
    \]
	It is easy to see that \(g(\omega) \coloneqq \int_0 ^\tau \mu\left( s \frac{\Pi_i(\mathbf{K} ^\omega)}{\omega_i} \in \mathbb{A}  \right)ds\) is constant on each equivalence class, since \(\mathbb{A} \) is shift invariant, so from~\ref{lem:HG} the expression above can be seen as a mixture of measures with weights given by \(\frac{H(d \omega)}{\theta_i \omega_i} \). 
    If we know that we have a finite (or countable) number \(c\) of different equivalence classes, we consider the partition \(\Sigma_d = I_1 \cup \dots \cup  I_c\) into disjoint union of elements of each class, and we obtain 
    \[
        \eta_i (\tilde{A}) = \theta_i \sum_{l= 1}^c \left( \int_{I_l} \frac{H(d \omega)}{\theta _i\omega_{i}}\right)\int_0 ^\tau \mu \left( s \hat{\mathbf{K}^l} \in \mathbb{A}   \right)ds   
    \]
	where \(\hat{\mathbf{K}^l} \) is a (normalized) representative of the equivalent class of \(\frac{\Pi_i(\mathbf{K} ^\omega)}{\omega_i} \) for \(\omega \in I_l\).
\end{remark}

\subsection{Multivariate exceedance point process}
Fix $\bbtau\in[0,\infty)^d$. Following~\cite{nandagopalan1994multivariate}, we define the
\emph{multivariate exceedance point process} as the $\mathbb{Z}_+^d$-valued random measure on
$[0,1]$ that counts, for each cluster $i$, how many times each component of $\mathbb{X}_{n,i}$
exceeds the threshold $\bbtau$. Precisely, for $J\in\mathcal{B}([0,1])$,
\[
    N_n^{(\bbtau)}(J) = \sum_{i=1}^n \delta_{(i/k_n)}(J)\,\boldsymbol{\varepsilon}_i,
    \qquad
    \boldsymbol{\varepsilon}_i = \Big(
        \sum_{j\in \mathbb{Z}}\mathbbm{1}_{\{|\tilde\pi(\mathbb{X}_{n,i})_{j,1}|<\tau_1\}},
        \dots,
        \sum_{j\in \mathbb{Z}}\mathbbm{1}_{\{|\tilde\pi(\mathbb{X}_{n,i})_{j,d}|<\tau_d\}}
    \Big),
\]
To identify the limit, we encode each cluster $x\in\tilde{l}^d_\infty$ by its vector of
exceedance counts. Define the \emph{multiplicity map}
$m\colon \tilde{l}_\infty^d \to \mathbb{Z}_+^d$ componentwise by
\[
    m_k(x) \coloneqq \sum_{j\in\mathbb{Z}}\mathbbm{1}_{\{|\pi_k(x_j)|<\tau_k\}},
    \qquad k=1,\dots,d,
\]
so that $\boldsymbol{\varepsilon}_i = m\big(\tilde\pi(\mathbb{X}_{n,i})\big)$, and lift it to point
processes by
\[
    \Phi\Big(\sum_i\delta_{(t_i,x_i)}\Big) = \sum_{i\,:\,m(x_i)\neq\mathbf{0}} \delta_{(t_i,\,m(x_i))},
\]
regarded as a map
$\mathcal{N}^\#_{\mathbb{R}^0_+\times\tilde{l}^d_\infty\setminus\{\infty\}}
\to \mathcal{N}_{\mathbb{R}^0_+\times(\mathbb{Z}_+^d\setminus\{\mathbf{0}\})}$.
Discarding the atoms with zero multiplicity is what guarantees local finiteness of the image. With
$N_n = \sum_i \delta_{(i/k_n,\,\tilde\pi(\mathbb{X}_{n,i}))}$ one has
\[
    N_n^{(\bbtau)}(J)
    = \int_{J\times(\mathbb{Z}_+^d\setminus\{\mathbf{0}\})} \mathbf{y}\,\Phi(N_n)(dt,d\mathbf{y}).
\]
The map $m$ is discontinuous at $x$ only when some coordinate $|\pi_k(x_j)|$ equals $\tau_k$
exactly. Under $N_\infty = \sum_i\delta_{(T_i,U_i\tilde{\mathbf{K}}_i)}$ this event has probability
zero, i.e.\ $\mu(N_\infty\in D_\Phi)=0$, so the continuous mapping theorem gives
\[
    \Phi(N_n)\;\longrightarrow\;\Phi(N_\infty),
    \qquad\text{hence}\qquad
    N_n^{(\bbtau)}\;\longrightarrow\; N_\infty^{(\bbtau)},
\]
where
\[
    N_\infty^{(\bbtau)}(J)
    = \int_{J\times(\mathbb{Z}_+^d\setminus\{\mathbf{0}\})} \mathbf{y}\,\Phi(N_\infty)(dt,d\mathbf{y})
    = \sum_{T_i\in J}\mathbf{m}_i,
    \qquad \mathbf{m}_i = m(U_i\tilde{\mathbf{K}}_i).
\]
Since $(T_i, U_i\tilde{\mathbf{K}}_i)_i$ is Poisson, $\Phi(N_\infty)=\sum_i\delta_{(T_i,\mathbf{m}_i)}$
is a PRM on $\mathbb{R}^0_+\times(\mathbb{Z}_+^d\setminus\{\mathbf{0}\})$ with intensity
$\mathrm{Leb}\times\lambda$, where
$\lambda(\cdot)=\Gamma(\bbtheta)\,\mu\big(m(U\tilde{\mathbf{K}})\in\cdot\bigm| m(U\tilde{\mathbf{K}})\neq\mathbf{0}\big)\,
\mu(m(U\tilde{\mathbf{K}})\neq\mathbf{0})$; equivalently, its total mass is
\[
    \Lambda
    = \Gamma(\bbtheta)\,\mu\big(m(U\tilde{\mathbf{K}})\neq\mathbf{0}\big)
    = \Gamma(\bbtheta)\int_{\Sigma_d}\!\int_0^\infty
        \mathbbm{1}_{\{sK^\omega\,\not\ge\,\bbtau\}}\,ds\,Q(d\omega)
    = \int_{\Sigma_d}\max_k\Big(\tfrac{\tau_k}{\omega_k}\Big)\,H(d\omega)
    = \Gamma(\bbtheta\bbtau),
\]
where we used $\inf_{j}|\pi_k(K^\omega_j)|=\omega_k$ (as in Corollary~\ref{cor:proj}),
$H=\Gamma(\bbtheta)\,Q$, and the spectral representation~\eqref{eq:spectral-representation}.
Setting
\[
    \hat{m}(\cdot)
    \coloneqq \mu\big(m(U\tilde{\mathbf{K}})\in\cdot \bigm| m(U\tilde{\mathbf{K}})\neq\mathbf{0}\big),
\]
its Laplace transform takes the Compound Poisson form: for $\mathbf{s}\in\mathbb{R}_+^d$,
\[
    \mathbb{E}\Big(e^{-\mathbf{s}'N_\infty^{(\bbtau)}(J)}\Big)
    = \exp\!\Big(
        -|J|\,\Gamma(\bbtheta\bbtau)
        \Big(1-\!\!\sum_{\mathbf{y}\in\mathbb{Z}_+^d\setminus\{\mathbf 0\}}\!\!
        e^{-\mathbf{s}'\mathbf{y}}\,\hat{m}(\mathbf{y})\Big)
    \Big).
\]
This recovers the Compound Poisson characterisation of~\cite{nandagopalan1994multivariate} with intensity $\Gamma(\bbtheta\bbtau)$
and multiplicity distribution given by $\hat m$.

\nocite{*}
\printbibliography

@article{aimino2026multivariate,
  doi = {10.1088/1361-6544/ae36db}, 
  url = {https://doi.org/10.1088/1361-6544/ae36db},
  year = {2026},
  publisher = {IOP Publishing},
  volume = {39},
  number = {1},
  pages = {015023},
  author = {Aimino, Romain and Cristina Moreira Freitas, Ana and Milhazes Freitas, Jorge and Todd, Mike},
  title = {Multivariate extreme values for dynamical systems},
  journal = {Nonlinearity},
}

@book{beirlant2006statistics,
  title     = {Statistics of extremes: theory and applications},
  author    = {Beirlant, Jan and Goegebeur, Yuri and
               Segers, Johan and Teugels, Jozef L.},
  year      = {2004},
  publisher = {John Wiley \& Sons}
}

@article{basrak2021compound,
  title     = {Compound Poisson approximation for regularly varying fields with application to sequence alignment},
  author    = {Basrak, Bojan and Planini{\'c}, Hrvoje},
  journal   = {Bernoulli},
  volume    = {27},
  number    = {2},
  pages     = {1371--1408},
  year      = {2021},
  doi       = {10.3150/20-BEJ1278}
}

@article{basrak2018invariance,
  title     = {An invariance principle for sums and record times of regularly varying stationary sequences},
  author    = {Basrak, Bojan and Planini{\'c}, Hrvoje and
               Soulier, Philippe},
  journal   = {Probability Theory and Related Fields},
  volume    = {172},
  number    = {3--4},
  pages     = {869--914},
  year      = {2018},
  publisher = {Springer},
  doi       = {10.1007/s00440-017-0822-9}
}

@article{de1977limit,
  title     = {Limit theory for multivariate sample extremes},
  author    = {De Haan, Laurens and Resnick, Sidney I.},
  journal   = {Zeitschrift f{\"u}r Wahrscheinlichkeitstheorie und verwandte Gebiete},
  volume    = {40},
  pages     = {317--337},
  year      = {1977},
  publisher = {Springer}
}

@article{freitas2020complete,
  author  = {Freitas, Ana Cristina Moreira and
             Freitas, Jorge Milhazes and
             Magalh{\~a}es, M{\'a}rio},
  title   = {Complete convergence and records for dynamically
             generated stochastic processes},
  journal = {Transactions of the American Mathematical Society},
  volume  = {373},
  number  = {1},
  year    = {2020},
  pages   = {435--478},
  mrnumber = {4042881}
}

@article{freitas2025enrichedfunctionallimittheorems,
  author  = { Freitas, Ana Cristina Moreira and
             Freitas, Jorge Milhazes and
             Todd, Mike},
  title   = {Enriched functional limit theorems for dynamical systems},
  journal = {Annali della Scuola Normale Superiore di Pisa - Classe di Scienze},
  year    = {2025},
  volume  = {59},
  doi     = {10.2422/2036-2145.202401_016}
}

@book{kulik2020heavy,
  title     = {Heavy-tailed time series},
  author    = {Kulik, Rafal and Soulier, Philippe},
  year      = {2020},
  publisher = {Springer}
}

@book{lucarini2016extremes,
  title     = {Extremes and recurrence in dynamical systems},
  author    = {Lucarini, Valerio and
               Faranda, Davide and
               Freitas, Ana Cristina Moreira and
               Freitas, Jorge Milhazes and
               Holland, Mark and
               Kuna, Tobias and
               Nicol, Matthew and
               Todd, Mike and
               Vaienti, Sandro},
  series    = {Pure and Applied Mathematics: A Wiley Series of Texts, Monographs and Tracts},
  publisher = {Wiley},
  address   = {Hoboken, NJ},
  year      = {2016}
}

@article{nandagopalan1994multivariate,
  title   = {On the multivariate extremal index},
  author  = {Nandagopalan, S.},
  journal = {Journal of Research of the National Institute of Standards and Technology},
  volume  = {99},
  year    = {1994},
  pages   = {543--550}
}

@article{segers2012max,
  author    = {Segers, Johan},
  title     = {Max-stable models for extremal dependence},
  journal   = {REVSTAT -- Statistical Journal},
  volume    = {10},
  number    = {1},
  pages     = {61--82},
  year      = {2012}
}

@article{pedersen2018tail,
  title     = {On the tail behavior of a class of multivariate conditionally heteroskedastic processes},
  author    = {Pedersen, Rasmus S. and Wintenberger, Olivier},
  journal   = {Extremes},
  volume    = {21},
  number    = {2},
  pages     = {261--284},
  year      = {2018},
  publisher = {Springer},
  doi       = {10.1007/s10687-017-0307-3}
}

@book{resnick1987extreme,
  title     = {Extreme values, regular variation and point processes},
  author    = {Resnick, Sidney I. and others},
  volume    = {4},
  year      = {1987},
  publisher = {Springer}
}

@misc{kirsebom2026visitingtimestatistics,
  title         = {Visiting time statistics},
  author        = {Kirsebom, Maxim and Kunde, Philipp and
                   Persson, Tomas},
  year          = {2026},
  eprint        = {2608.01453},
  archivePrefix = {arXiv},
  primaryClass  = {math.DS},
  url           = {https://arxiv.org/abs/2608.01453}
}

@article{pene2020spatio,
  title     = {Spatio-temporal Poisson processes for visits to small sets},
  author    = {P{\`e}ne, Fran{\c{c}}oise and Saussol, Beno{\^i}t},
  journal   = {Israel Journal of Mathematics},
  volume    = {240},
  pages     = {625--665},
  year      = {2020},
  publisher = {Springer}
}

@article{zweimuller2022hitting,
  title   = {Hitting times and positions in rare events},
  author  = {Zweim{\"u}ller, Roland},
  journal = {Annales Henri Lebesgue},
  volume  = {5},
  pages   = {1361--1415},
  year    = {2022}
}

@article{faranda2018correlation,
  title     = {Correlation dimension and phase space contraction via extreme value theory},
  author    = {Faranda, Davide and Vaienti, Sandro},
  journal   = {Chaos: An Interdisciplinary Journal of Nonlinear Science},
  volume    = {28},
  number    = {4},
  year      = {2018},
  publisher = {AIP Publishing}
}

@article{bansard2026fractional,
  title     = {The fractional Poisson process and other limit point processes for rare events in infinite ergodic theory},
  author    = {Bansard-Tresse, Dylan},
  journal   = {Advances in Mathematics},
  volume    = {492},
  pages     = {110916},
  year      = {2026},
  publisher = {Elsevier},
  doi       = {10.1016/j.aim.2026.110916}
}
\end{document}